\documentclass[a4paper, 11pt, english]{article}
\usepackage[utf8]{inputenc}
\usepackage[T1]{fontenc}
\usepackage{babel}
\usepackage{graphicx}
\usepackage{tikz}
\usetikzlibrary{hobby,patterns}
\usepackage{stmaryrd}
\usepackage[a4paper]{geometry}
\usepackage{amsmath,amsfonts,amssymb,amsthm,epsfig,epstopdf,url,array}
\usepackage{rotating}
\usepackage[colorlinks=true,citecolor=red,linkcolor=blue,pdfpagetransition=Blinds]{hyperref}
\usepackage{cleveref}
\usepackage{nameref}
\usepackage{enumitem}
\usepackage{comment}
\Crefname{paragraph}{Section}{Sections}
\usepackage{fancyhdr}

\usepackage{fullpage}

\newcommand{\ensemblenombre}[1]{\mathbb{#1}}
\newcommand{\N}{\ensemblenombre{N}}

\newcommand{\R}{} 
\renewcommand{\R}{\ensemblenombre{R}}
\newcommand{\C}{\ensemblenombre{C}}

\renewcommand{\leq}{\leqslant}
\renewcommand{\geq}{\geqslant}

\newcommand{\norme}[1]{\left\lVert#1\right\rVert}

\newcommand{\dive}[1]{\mathrm{div}}

\theoremstyle{plain} 
\newtheorem{prop}{Proposition}[section] 
\newtheorem{theorem}[prop]{Theorem}

\newtheorem{lemma}[prop]{Lemma}
\newtheorem{cor}[prop]{Corollary}
\theoremstyle{definition}

\numberwithin{equation}{section}

\makeatletter
\let\original@addcontentsline\addcontentsline
\newcommand{\dummy@addcontentsline}[3]{}
\newcommand{\DeactivateToc}{\let\addcontentsline\dummy@addcontentsline}
\newcommand{\ActivateToc}{\let\addcontentsline\original@addcontentsline}
\makeatother

\begin{document}

\title{Quantitative unique continuation for real-valued solutions to second order elliptic equations in the plane}
\author{Kévin Le Balc'h \footnote{Inria, Sorbonne Université, CNRS, Laboratoire Jacques-Louis Lions, Paris, France. {\tt kevin.le-balc-h@inria.fr}  }\ , Diego A. Souza  \footnote{University of Sevilla, Dpto. EDAN, Aptdo 1160, 41080 Sevilla, Spain. {\tt desouza@us.es}}
}

\maketitle
\begin{abstract}
In this article, we study a quantitative form of the Landis conjecture on exponential decay for real-valued solutions to second order elliptic equations with variable coefficients in the plane. In particular, we prove the following qualitative form of Landis conjecture, for $W_1, W_2 \in L^{\infty}(\R^2;\R^2)$, $V \in L^{\infty}(\R^2;\R)$ and $u \in H_{\mathrm{loc}}^{1}(\R^2)$ a real-valued weak solution to $-\Delta u - \nabla \cdot ( W_1 u )  +W_2 \cdot \nabla u  + V u = 0$ in $\R^2$, satisfying for $\delta>0$, $|u(x)| \leq \exp(- |x|^{1+\delta})$, $x \in \R^2$, then $u \equiv 0$. Our methodology of proof is inspired by the one recently developed by Logunov, Malinnikova, Nadirashvili, and Nazarov that have treated the equation $-\Delta u + V u = 0$ in $\R^2$. Nevertheless, several differences and additional difficulties appear. New weak quantitative maximum principles are established for the construction of a positive multiplier in a suitable perforated domain, depending on the nodal set of $u$. The resulted divergence elliptic equation is then transformed into a non-homogeneous $\partial_{\overline{z}}$ equation thanks to a generalization of Stoilow factorization theorem obtained by the theory of quasiconformal mappings, an approximate type Poincaré lemma and the use of the Cauchy transform. Finally, a suitable Carleman estimate applied to the operator $\partial_{\overline{z}}$ is the last ingredient of our proof.
\end{abstract}

\tableofcontents

\section{Introduction}

\subsection{Qualitative and quantitative unique continuation at infinity}

 In the late 1960's, see \cite{KL88}, Landis conjectured the following: for $V \in L^{\infty}(\R^N)$ and $\delta>0$,
\begin{equation}
\label{eq:landisconjecture}
				(-\Delta u + V u = 0  \text{ in }\R^N\ \text{and}\ 
				|u(x)| \leq  \exp(-|x|^{1+\delta}) \text{ in }\R^N)
		\quad \Rightarrow \quad u \equiv 0 \text{ in } \R^N.
\end{equation}
One can see \eqref{eq:landisconjecture} as a \textit{qualitative unique continuation} property at infinity. The decay rate $\exp(-|x|^{1+\delta})$ seems to be a natural barrier, by considering the function $u(x) = \exp( - C \sqrt{1+|x|^2})$ for a suitable constant $C>0$. Moreover, \eqref{eq:landisconjecture} holds when $N=1$ by an ordinary differential argument, see for instance \cite{Ros18} or \cite[Introduction]{LB21}.

Landis conjecture was first disproved by Meshkov in 1991 in the case of complex-valued potentials $V$. In fact, the work \cite{Mes91} exhibits in the plane $\R^2$ a counterexample to \eqref{eq:landisconjecture}:
\begin{equation}
\label{eq:meshkovcounterexample}
\exists V \in L^{\infty}(\R^2;\C)\ \text{and}\ u \not\equiv 0,\
				-\Delta u + V u = 0  \text{ in }\R^2\ \text{and}\ 
				|u(x)| \leq \exp(-|x|^{4/3})  \text{ in }\R^2.
\end{equation}
\cite{Mes91} also shows that this is the right scale, proving the qualitative unique continuation property at infinity: for $V \in L^{\infty}(\R^N)$ and $\delta>0$, we have
\begin{equation}
\label{eq:meshkovdecay}
(-\Delta u + V u = 0  \text{ in }\R^N\ \text{and}\ 
				|u(x)| \leq \exp(-|x|^{4/3+\delta}) \text{ in }\R^N)
		\quad \Rightarrow \quad u \equiv 0 \text{ in } \R^N.
\end{equation}
In their work on Anderson localization \cite{BK05}, Bourgain and Kenig establish a \textit{quantitative version} of Meshkov's result, that is assuming that $-\Delta u +  V u = 0$ in $\R^N$, with $\|V\|_{\infty} \leq 1$ and $\|u\|_{\infty} = |u(0)|=1$, then for $C,C'>0$ sufficiently large
\begin{equation}
\label{eq:bourgainkenig}
\sup_{B(x_0,1)} |u(x)| \geq \exp(-C R^{4/3}\log(R))\qquad\forall R \geq C',\  \forall x_0 \in \R^N\ \text{with}\ |x_0| = R.
\end{equation}

The case of real-valued potentials has been addressed in \cite{BK05} and is more  tricky. We may first ask if the qualitative Landis conjecture \eqref{eq:landisconjecture} holds for real-valued bounded potentials $V$. Then, we may wonder if the quantitative Landis conjecture holds for real-valued potentials, i.e. if \eqref{eq:bourgainkenig} holds replacing $4/3$ by $1$. The difficulty for tackling such a question comes from the fact that Carleman estimates do not seem to distinguish between real-valued and complex-valued solutions to elliptic equations.

A first breakthrough was achieved in \cite{KSW15}, regarding the quantitative unique continuation at infinity in the plane. Assuming that $-\Delta u - \nabla \cdot ( W u) +  V u = 0$ in $\R^2$ or $-\Delta u + W \cdot \nabla u +  V u = 0$ in $\R^2$, with $W \in L^{\infty}(\R^2;\R)$, $\| W \| \leq 1$, $V \in L^{\infty}(\R^2;\R)$, $0 \leq V \leq 1$ and $\|u\|_{\infty} = |u(0)|=1$, then for $C,C'>0$ sufficiently large
\begin{equation}
\label{eq:kenigsylvestrewang}
\sup_{B(x_0,1)} |u(x)| \geq \exp(-C R \log(R))\qquad\forall R \geq C',\  \forall x_0 \in \R^2\ \text{with}\ |x_0| = R.
\end{equation}
Then, subsequent papers established analogous results in the settings of variable coefficients and singular lower-order terms, \cite{KW15,Dav20,DW20}, always assuming a sign condition on the zero order term $V$.

A second breakthrough was achieved very recently in the $2$-d case in the work \cite{LMNN20} by withdrawing the sign condition on the potential $V$, proving in particular \eqref{eq:landisconjecture} in the real-valued case. More precisely, the authors prove that for $V \in L^{\infty}(\R^2;\R)$, there exists $C>0$ sufficiently large such that
\begin{equation}
\label{eq:landisconjectureLog}
(-\Delta u + V u = 0  \text{ in }\R^2\ \text{and}\ 
				|u(x)| \leq \exp(-C|x| \log^{1/2}(1+|x|)) \text{ in }\R^2)
		\quad \Rightarrow \quad u \equiv 0 \text{ in } \R^2.
\end{equation}
Actually, the authors prove the following quantitative unique continuation at infinity. Assuming that $-\Delta u + V u = 0 \in \R^2$, with $-1 \leq V \leq 1$ and $\|u\|_{\infty} = |u(0)|=1$, then for $C,C'>0$ sufficiently large
\begin{equation}
\label{eq:lmnn}
\sup_{B(x_0,1)} |u(x)| \geq \exp(-C R \log^{3/2}(R)) \qquad\forall R \geq C',\  \forall x_0 \in \R^2\ \text{with}\ |x_0| = R.
\end{equation}

Based on the new idea coming from \cite{LMNN20}, the goal of this article is to give a positive answer to the Landis conjecture \eqref{eq:landisconjecture} for real-valued solutions to elliptic equations $-\Delta u - \nabla \cdot ( W_1 u )  +W_2 \cdot \nabla u  + V u = 0$ in $\R^2$ with $W_1, W_2 \in L^{\infty}(\R^2;\R^2)$ and $V \in L^{\infty}(\R^2;\R)$ and prove a quantitative version of the Landis conjecture in the spirit of \eqref{eq:bourgainkenig}, \eqref{eq:kenigsylvestrewang}, \eqref{eq:lmnn}.

\subsection{Main results}

The first main result of this paper is the following positive answer to the qualitative Landis conjecture in the plane for real-valued solutions to second order elliptic equations.
\begin{theorem}
\label{Thm:Landis}
Let $u \in H_{\mathrm{loc}}^1(\R^2)$ be a real-valued weak solution to 
\begin{equation}
\label{eq:equationuLandis}
-\Delta u - \nabla \cdot ( W_1 u )  +W_2 \cdot \nabla u  + V u = 0\ \text{in}\ \R^2,
\end{equation}
with
\begin{equation}
\label{eq:boundlowerlandis}
W_1, W_2 \in L^{\infty}(\R^2;\R^2),\ V \in L^{\infty}(\R^2;\R).
\end{equation}
Assume that there exists $\delta >0$ such that
\begin{equation}
\label{eq:decayexpu}
|u(x)| \leq \exp(- |x|^{1+\delta})\qquad \forall x \in \R^2.
\end{equation}
Then, $u \equiv 0$.
\end{theorem}

Our second main result is the following quantitative unique continuation property at infinity.
\begin{theorem}
\label{Thm:QuantitativeLandis}
Let $u \in H_{\mathrm{loc}}^1(\R^2) \cap L^{\infty}(\R^2)$ be a real-valued weak solution to 
\begin{equation}
\label{eq:equationuLandisquantitative}
-\Delta u - \nabla \cdot ( W_1 u )  +W_2 \cdot \nabla u  + V u = 0\ \text{in}\ \R^2,
\end{equation}
with
\begin{equation}
\label{eq:boundlowerlandisquantitative}
W_1, W_2 \in L^{\infty}(\R^2;\R^2),\ V \in L^{\infty}(\R^2;\R),\ \norme{W_1}_{\infty} \leq 1,\ \norme{W_2}_{\infty} \leq 1,\ \norme{V}_{\infty} \leq 1.
\end{equation}
Assume that 
\begin{equation}
\|u\|_{\infty} = |u(0)|=1.
\end{equation}
Then, for every $\delta >0$, there exist a positive constant $C=C(\delta)\geq 1$, such that
\begin{equation}
\label{eq:observabilityuplane}
\sup_{B(x_0,1)} |u(x)| \geq \exp(-C R^{1 + \delta})\qquad \forall R \geq 2,\  \forall x_0 \in \R^2\ \text{with}\ |x_0| = R.
\end{equation}
\end{theorem}

\Cref{Thm:Landis} and \Cref{Thm:QuantitativeLandis} are actually based on local quantitative unique continuation properties and a scaling argument that we present in the next part.

\subsection{Local quantitative unique continuation property}
For the next, we introduce the notation $B_r = B(0,r)$ for $r>0$ and $\log_+(s) = \log(2+s)$ for $s \geq 0$.

The following result relates on the vanishing order of real-valued solutions to second order elliptic equations.
\begin{theorem}
\label{thm:landislocal}
Let $u \in H_{\mathrm{loc}}^1(B_2) \cap L^{\infty}(B_2)$ be a real-valued weak solution to 
\begin{equation}
\label{eq:equationuLandisquantitativelocal}
-\Delta u - \nabla \cdot ( W_1 u )  +W_2 \cdot \nabla u  + V u = 0\ \text{in}\ B_2,
\end{equation}
with
\begin{equation}
\label{eq:boundlowerlandisquantitativelocal}
W_1, W_2 \in L^{\infty}(B_2;\R^2),\ V \in L^{\infty}(B_2;\R).
\end{equation}
Assume that for $K \geq 2$,
\begin{equation}
\label{eq:ucontrolledneartheboundary}
\norme{u}_{L^{\infty}(B_2)} \leq e^{K} \norme{u}_{L^{\infty}(B_{1})}.
\end{equation}
Then, for every $\delta >0$, there exists a positive constant $C = C(\delta) \geq 1$ such that
\begin{equation}
\label{eq:observabilityu}
\norme{u}_{L^{\infty}(B_r)} \geq r^{C \left(\|W_1\|_{\infty}^{1 + \delta} + \|W_2\|_{\infty}^{1 + \delta} + \|V\|_{\infty}^{1/2} \log_{+}^{3/2}(\|V\|_{\infty})\right)+ C K + C } \norme{u}_{L^{\infty}(B_2)}\quad  \forall r \in (0,1/2).
\end{equation}
\end{theorem}

The rescaled version of \Cref{thm:landislocal} is the following result.
\begin{theorem}
\label{thm:landislocalR}
Let $R \geq 2$. Let $u \in H_{\mathrm{loc}}^1(B_{2R}) \cap L^{\infty}(B_{2R})$ be a real-valued weak solution to 
\begin{equation}
\label{eq:equationuLandisquantitativeR}
-\Delta u - \nabla \cdot ( W_1 u )  +W_2 \cdot \nabla u  + V u = 0\ \text{in}\ B_{2R},
\end{equation}
with
\begin{equation}
\label{eq:boundlowerlandisquantitativeR}
W_1, W_2 \in L^{\infty}(B_{2R};\R^2),\ V \in L^{\infty}(B_{2R};\R),\ \norme{W_1}_{\infty} \leq 1,\ \norme{W_2}_{\infty} \leq 1,\ \norme{V}_{\infty} \leq 1.
\end{equation}
Assume that for $K \geq 2$,
\begin{equation}
\label{eq:ucontrolledneartheboundaryR}
\norme{u}_{L^{\infty}(B_{2R})} \leq e^{K} \norme{u}_{L^{\infty}(B_{R})}.
\end{equation}
Then, for every $\delta >0$, there exists a positive constant $C = C(\delta) \geq 1$ such that
\begin{equation}
\label{eq:observabilityuR}
\norme{u}_{L^{\infty}(B_{r})} \geq \left(r/R\right)^{C R^{1+\delta} + C K }  \norme{u}_{L^{\infty}(B_{2R})}\qquad \forall r \in (0,R/2).
\end{equation}
\end{theorem}
\medskip
The end of this part consists in proving the following sequence of implications:
\begin{equation}
\label{eq:implicationtheorem}
\text{\Cref{thm:landislocal}} \Rightarrow  \text{\Cref{thm:landislocalR}} \Rightarrow \text{\Cref{Thm:Landis}}\ \text{and}\ \text{\Cref{Thm:QuantitativeLandis}}.
\end{equation} 

\begin{proof}[Proof of \Cref{thm:landislocalR} from \Cref{thm:landislocal}]
We apply \Cref{thm:landislocal} to $u_R(\cdot) = u(R\cdot)$ that solves \eqref{eq:equationuLandisquantitativelocal} with $W_{1,R} = R W_1(R \cdot)\in L^{\infty}(B_2;\R^2),\ W_{2,R} = R W_2(R \cdot)\in L^{\infty}(B_2;\R^2)\ \text{and}\ V_R = R^2 V(R \cdot) \in L^{\infty}(B_2;\R).$
Remark that 
$$\norme{W_{1,R}}_{\infty} \leq R,\ \norme{W_{2,R}}_{\infty} \leq R,\ \text{and}\ \norme{V}_{\infty} \leq R^2,$$ 
so for every $r \in (0,R/2)$, that is $(r/R) \in (0,1/2)$, we have
\begin{equation*}
\norme{u}_{L^{\infty}(B_r}= \norme{u_R}_{L^{\infty}(B_{r/R})} \geq (r/R)^{C R^{1 + \delta} + C K} \norme{u_R}_{L^{\infty}(B_2)} 
\geq (r/R)^{C R^{1 + \delta} + C K } \norme{u}_{L^{\infty}(B_{2R})},
\end{equation*}
leading to the expected inequality \eqref{eq:observabilityuR}.
\end{proof}
We now prove \Cref{Thm:Landis} and \Cref{Thm:QuantitativeLandis} from \Cref{thm:landislocalR}.
\begin{proof}[Proof of \Cref{Thm:Landis} from \Cref{thm:landislocalR}]
Replacing $u$ by $u_{\lambda}(\cdot) = u(\lambda \cdot)$ for $\lambda >0$ small enough, one can assume that $\norme{W_1}_{\infty} \leq 1,\ \norme{W_2}_{\infty} \leq 1,\ \norme{V}_{\infty} \leq 1$. We then argue by contradiction, assuming that $u_{\lambda}$ is not identically equal to $0$. By using that $|u_{\lambda}|$ tends to $0$ near infinity, we have that $|u_{\lambda}|$ attains its global maximum at some point $x_{\max}$ on the plane. Then, for any $R \geq 2 |x_{\max}| + 2 $ and any $x$ with $|x| = R/2$, we have
$$ \sup_{B(x,2R)} |u_{\lambda}|= \sup_{B(x,R)} |u_{\lambda}|,$$
and additionally by applying \Cref{thm:landislocalR} to $u_{\lambda}(x+\cdot)$ with $\delta/2$, we have for $C \geq 1$,
$$ \sup_{B(x,R/4)} |u_{\lambda}| \geq \exp(-C R^{1+\delta/2} ),$$
leading to a contradiction with the decaying assumption \eqref{eq:decayexpu}.
\end{proof}
\begin{proof}[Proof of \Cref{Thm:QuantitativeLandis} from \Cref{thm:landislocalR}]
Take $x_0 \in \R^2$ such that $|x_0| = R$, then from the assumption, $\|u\|_{\infty} = |u(0)|=1$ we have $$\norme{u(x_0+\cdot)}_{L^{\infty}(B_{2R})}= \norme{u(x_0+\cdot)}_{L^{\infty}(B_R)},$$ so one can apply \eqref{eq:observabilityuR} to the function $u(x_0+\cdot)$ with $r=1 \leq R/2$ to get 
$$ \norme{u(x_0+\cdot)}_{L^{\infty}(B_1)} \geq (1/R)^{C R^{1 + \delta/2}} \geq \exp(-C R^{1+ \delta/2} \log(R) ) \geq \exp(-C R^{1+ \delta}),$$
so \eqref{eq:observabilityuplane} holds.
\end{proof}


\subsection{Strategy of the proof of the main local result \Cref{thm:landislocal}}
\label{sec:strategyofproof}

\textbf{Notation and parameters.} In the following and in the whole paper, $C,C'\geq 1$ denote various positive large numerical constants, $c,c'>0$ denote various positive small numerical constants and $\varepsilon>0$ is a free sufficiently small parameter that would be chosen depending on $\norme{W_1}_{\infty}$, $\norme{W_2}_{\infty}$,   $\norme{V}_{\infty}$, see below. During the proof, we need to adjust or to fix some constants or parameters, this would be precisely indicated in paragraphs called “\textbf{Setting of parameters}”.\\

In this part, we present the strategy of the proof of \Cref{thm:landislocal} and the main arguments of each step. 
This strategy actually follows the approach of \cite{LMNN20}. We will explain at the end of this section the new difficulties in comparison to \cite{LMNN20}. The proof of \Cref{thm:landislocal} is divided into three main steps.

\begin{itemize}
\item \textbf{Step 1: Construction of a positive multiplier $\varphi$ in a suitable perforated domain.}
We first introduce the set of zeros of $u$, called the \textit{nodal set} of $u$,
\begin{equation}
Z := \{x \in B_2\ ;\ u(x) = 0\}.
\end{equation}
In this step, we shall first prove that $Z$ satisfies the following fundamental property
\begin{equation}
\label{eq:Pepsilon}
\forall x_0 \in Z,\ \forall \rho \in (0,\varepsilon),\ \partial B(x_0,\rho) \cap (Z \cup \partial B(0,2)) \neq \emptyset, 
\tag{P-$\varepsilon$}
\end{equation}
for 
\begin{equation}
\label{eq:firstvarepsilon}
\varepsilon \leq c +  c \|W_1\|_{\infty}^{-1-\delta/2} + c \|W_2\|_{\infty}^{-1-\delta/2} +  c \|V\|_{\infty}^{-1/2}.
\end{equation}
The next point consists in \textit{perforating the domain} $B(0,2)$ using sufficiently small disks (of radius $\varepsilon$) in a sufficiently large number whose union is denoted by $F_{\varepsilon}$, avoiding $Z$, $\partial B(0,2)$, $0$ and $x_{\max}$, the point at which $|u|$ is maximal in $B_{1}$. The resulting perforated domain $\Omega_{\varepsilon} = B_2 \setminus (Z \cup F_{\varepsilon})$ has a \textit{small Poincaré constant} of the form $C' \varepsilon$ so one can construct a positive solution $\varphi \in H^1(\Omega_{\varepsilon})$ 
\begin{equation}
		\label{Eq-varphiOepsProof}
				-\Delta \varphi - \nabla \cdot (W_2  \varphi) + W_1 \cdot \nabla \varphi + V \varphi = 0\ \text{ in } \Omega_{\varepsilon}, 	
	\end{equation}
and
\begin{equation}
\label{eq:boundmultiplier}
\varphi-1 \in H_0^1(\Omega_{\varepsilon}),\ \norme{\varphi -1}_{L^{\infty}(\Omega_{\varepsilon})} \leq C \varepsilon^{2/(2+\delta) } \norme{W_2}_{\infty} +  C \varepsilon^2 \norme{V}_{\infty}.
\end{equation}
In the following, we will call this solution $\varphi$ a \emph{multiplier}. Note that for the construction of the multiplier, $\varepsilon$ is still of the form \eqref{eq:firstvarepsilon}.

\item \textbf{Step 2: Reduction to a non-homogeneous $\partial_{\overline{z}}$ equation.} Thanks to the positive multiplier of the previous step, we first reduce the elliptic equation $-\Delta u - \nabla \cdot ( W_1 u )  +W_2 \cdot \nabla u  + V u = 0$ to a \textit{divergence type elliptic equation} satisfied by $v= u /\varphi$,
\begin{equation}
\label{eq:divergencellipticinsteps}
-\nabla \cdot( \varphi^2 (\nabla v +  \hat{W} v)) = 0\ \text{in}\ \Omega_{\varepsilon}' = B_2 \setminus F_{\varepsilon},\ \text{with}\  \hat{W} = W_1 - W_2.
\end{equation}
Note that the divergence elliptic equation is satisfied in the weak sense, through the nodal set of $u$. We then apply the \textit{theory of quasiconformal mappings} to find $L : B_2 \to B_2$, a quasiconformal mapping, to recast the divergence elliptic equation satisfied by $h = v \circ L^{-1}$,
\begin{multline}
\label{eq:equationhproofstrategy}
-\Delta h - \nabla \cdot ( \tilde{W} h) = 0\ \text{in}\ L(\Omega_{\varepsilon}'),\\ \text{with}\ \tilde{W} = \overline{\partial_{z} L^{-1}} \cdot \overset{\diamond}{W}\circ L^{-1},\ \|\overset{\diamond}{W}\|_{\infty} \leq \norme{W_1}_{\infty} + \norme{W_2}_{\infty}.
\end{multline}
The next point of this step consists in controlling how the quasiconformal change of variable, denoted by $L$ transforms $\Omega_{\varepsilon}'$ to another perforated domain. In particular, the holes, which were disks before, will be transformed into holes which still cannot be too flattened by this quasiconformal transform. Moreover, local $W^{1,p}$-estimates on $L^{-1}$ are also established to handle the extra term $\partial_z L^{-1}$ appearing in the definition of $\tilde{W}$. For this step, $\varepsilon$ has to be chosen such that 
\begin{equation}
\label{eq:secondvarepsilon}
\varepsilon \leq c + c \|W_2\|_{\infty}^{-1-\delta/2}\log_{+}^{-1-\delta/2} (\|W_2\|_{\infty}) + c \|V\|_{\infty}^{-1/2} \log_{+}^{-1/2}(\|V\|_{\infty}).
\end{equation}

We then introduce an \textit{approximate stream function} to the previous divergence free vector, i.e. $\nabla h + \tilde{W} h$, then use the \emph{Cauchy transform} that enables to recast the previous elliptic equation into a \textit{non-homogeneous reduced Beltrami equation}
\begin{equation}
\label{eq:equationdzbarstep2}
\partial_{\overline{z}} \zeta = F\ \text{in}\ B_2,
\end{equation}
where $F$ is a source term depending on the values of $v, \nabla v$ near the disks of the perforated domain $L(\Omega_{\varepsilon}')$. Note that at the end of this step, $\varepsilon$ is now fixed, satisfying both \eqref{eq:firstvarepsilon} and \eqref{eq:secondvarepsilon} then
\begin{equation}
	\label{eq:epsPetitPoincareGeneralFinStep2ProofStrategy}
	\varepsilon \leq c + c \|W_1\|_{\infty}^{-1-\delta/2} + c \|W_2\|_{\infty}^{-1-\delta/2} \log_{+}^{-1-\delta/2} (\|W_2\|_{\infty}) +  c \|V\|_{\infty}^{-1/2} \log_{+}^{-1/2}(\|V\|_{\infty}).
\end{equation}

\item \textbf{Step 3: A Carleman estimate to $\partial_{\overline{z}}$.} We now employ a \emph{Carleman estimate} in $B(0,2)$ to cut-off version of $\zeta$, called $y$, which vanishes in a small neighbourhood of $\partial B(0,2)$, in a $r'$-neighbourhood of $B(0,r'/2)$ where $B(0,r') \subset L(B(0,r))$, 
\begin{equation}
\label{eq:CarlemanyPresentation}
 \int_{B_2} |y|^2  e^{-2 s \log(|z|) + 2 |z|^2} dz \leq C \int_{B_2}|\partial_{\overline{z}}y|^2  e^{-2 s \log(|z|) + 2 |z|^2} dz\qquad \forall s \geq 1.
\end{equation}
By using \emph{Harnack inequalities}, the source term $F$ is then absorbed by taking the parameter $s$ in the Carleman estimate such that 
\begin{equation}
\label{eq:sfonctionepsilon}
s \geq C \varepsilon^{-1} \log(C \varepsilon^{-1}),
\end{equation}
so according to \eqref{eq:epsPetitPoincareGeneralFinStep2ProofStrategy} the following choice of $s$ is convenient
\begin{equation}
\label{eq:sparametersepsilon}
s \geq C \left(\|W_1\|_{\infty}^{1 + \delta} + \|W_2\|_{\infty}^{1 + \delta} + \|V\|_{\infty}^{1/2} \log_{+}^{3/2}(\|V\|_{\infty})\right) + C.
\end{equation}
The cut-off terms near $\partial B(0,2)$ are absorbed by using \eqref{eq:ucontrolledneartheboundary} and by recalling that the perforation process in Step 1 avoids the point $x_{\max}$, here $s$ has to be taken such that 
\begin{equation}
\label{eq:sfonctionK}
s \geq  C \left(\|W_1\|_{\infty}^{1 + \delta} + \|W_2\|_{\infty}^{1 + \delta} + \|V\|_{\infty}^{1/2} \log_{+}^{3/2}(\|V\|_{\infty})\right) + C K + C.
\end{equation}
The cut-off term near $B(0,r'/2)$ will be our observation term, i.e. the left hand side of \eqref{eq:observabilityu} recalling that $r' = c r^2$ if $r \leq \varepsilon$ or $r' = cr$ if $r > \varepsilon$. This combination of arguments leads to the expected quantitative unique continuation estimate for $u$, i.e. \eqref{eq:observabilityu}. 

\end{itemize}

Steps 1, 2 and 3 are crucially inspired by the methodology in \cite{LMNN20} that focuses on the case of the elliptic equation $-\Delta u+ Vu = 0$. Still, our strategy differs from the one in \cite{LMNN20} in several points. \\

\textit{Differences of Step 1 in comparison to \cite[Act 1]{LMNN20}.} The main difference is the presence of the drift terms $W_1, W_2$. 

We first prove a weak quantitative maximum principle for $\Phi \in H_0^1(\Omega)$ satisfying $-\Delta \Phi + W \cdot \nabla \Phi = f$ with $f \in L^{\infty}(\Omega)$, where $\Omega$ is a bounded open set with small Poincaré constant, see \Cref{lem:varphiTilde} below. This is a generalization of the weak quantitative maximum principle \cite[Lemma 6.10]{LMNN20} for the Laplace equation $-\Delta \Phi = f$. In \cite[Section 6.3]{LMNN20}, the authors use De Giorgi method conjugated with the fact that $\Phi$ is the minimizer of the functional $F(\Phi) = \int_{\Omega} |\nabla \Phi|^2 -  \int_{\Omega} f \Phi$ to establish \cite[Lemma 6.10]{LMNN20}. Here, because the operator $-\Delta + W \cdot \nabla$ is not symmetric, we have to proceed in another way. We instead implement De Giorgi method in the associated variational formulation of the elliptic equation $-\Delta \Phi + W \cdot \nabla \Phi = f$.

We then prove a weak quantitative maximum principle $\Phi \in H_0^1(\Omega)$ satisfying $-\Delta \Phi + W \cdot \nabla \Phi = \nabla \cdot g$ with $g \in L^{\infty}(\Omega)$, where $\Omega$ is a bounded open set with small Poincaré constant, see \Cref{lem:varphiTildeDivergenceSource} below. This part is new in comparison to \cite{LMNN20}. For establishing such a result, we first derive precise Sobolev's inequalities, quantified in function of the Poincaré constant, see \Cref{lem:sobolevinequalities} below. This enables us to follow Stampacchia's iterative strategy for the obtaining of the $L^{\infty}$-bound. It is worth mentioning that the $L^{\infty}$-bound coming from the maximum principle depends on the measure of the open set, that is not the case of the previous maximum principle.

On the one hand, the maximum principles with both $L^{\infty}$ and $W^{-1, \infty}$-source terms are useful for proving that the nodal set of $u$ satisfies \eqref{eq:Pepsilon}, by constructing appropriative positive functions $\varphi$ to $-\Delta \varphi - \nabla \cdot (W_1  \varphi) + W_2 \cdot \nabla \varphi + V \varphi=0$, see \Cref{lem:varphidiv} and \Cref{lem:nodalsetu} below. On the other hand, the maximum principles with both $L^{\infty}$ and $W^{-1, \infty}$-source terms lead to the construction of positive multiplier $\varphi \in H^1(\Omega)$ satisfying $-\Delta \varphi - \nabla \cdot (W_2  \varphi) + W_1 \cdot \nabla \varphi + V \varphi=0$ where $\Omega$ is a bounded open set with small Poincaré constant, that is analogous to \cite[Lemma 3.2]{LMNN20}, see \Cref{lem:varphidiv} and  \Cref{prop:constructionmultiplier} below. \\


\textit{Differences of Step 2 in comparison to \cite[Act 2]{LMNN20}.}
The philosophy of this step is the same as \cite[Act 2]{LMNN20}, we try to reduce the divergence elliptic equation  $-\nabla \cdot( \varphi^2 (\nabla v + \hat{W} v)) = 0$ in $\Omega_{\varepsilon}'$ to a very simple elliptic equation by using a quasi-conformal change of variable. But here the strategy is rather different. 

The problem is that one cannot reduce it to a simple harmonic equation because of the presence of the drift term $\hat{W}$. By working as in \cite[Act 2]{LMNN20}, one can define a local stream $\tilde{v}$ function, in an arbitrary ball $B$ contained in $\Omega_{\varepsilon}'$, by using Poincaré lemma. Then, the function $w = v + i \tilde{v}$ satisfies a Beltrami equation 
$$\partial_{\overline{z}} w = \mu \partial_{z} w - \overset{\diamond}{W} w\ \text{in}\ B \subset \Omega_{\varepsilon}',$$ 
with $$|\mu| \leq C \varepsilon^{2/(2+\delta) } \norme{W_2}_{\infty}+  C \varepsilon^2 \|V\|_{\infty}\ \text{and}\ \|\overset{\diamond}{W}\|_{\infty} \leq C \|W_1\|_{\infty} + C \|W_2\|_{\infty}.$$ One can then introduce a quasiconformal mapping $L : B_2 \to B_2$, satisfying $\partial_{\overline{z}} L = \mu \partial_z L$. An adaptation of the Stoilow factorization theorem then leads to the equation \eqref{eq:equationhproofstrategy} satisfied by $h = v \circ L^{-1}$, see \Cref{lem:harmonich} below. 

One needs to reduce a bit more because of the presence of the new drift term $\tilde{W}$. In order to do this, we first establish that $L^{-1} \in W_{\mathrm{loc}}^{1,p}(B_2)$ for every $1 \leq p < +\infty$ so that for $c_0>0$, $\|\tilde{W}\|_{L^{p}(B_{2-c_0})} \leq C \|\overset{\diamond}{W}\|_{\infty}$, see \Cref{lem:imagequasiconformal} below. We then introduce a cut-off $\chi$ near the images of the disks of the perforated domain, and near the boundary of $B_2$ to recast \eqref{eq:equationhproofstrategy} into a divergence elliptic equation, with a non-homogeneous source term
\begin{equation}
\label{eq:equationchivproofstrategy}
-\nabla \cdot(\chi  (\nabla h + \tilde{W} h)) = - \nabla \chi \cdot (\nabla  h + \tilde{W} h) \ \text{in}\ B_2,
\end{equation}
see \Cref{lem:divergencepcutoff} below. The advantage of such a procedure is that we are now working in the simply connected domain $B_2$ and one can then deduce an approximate type Poincaré lemma to define $\tilde{h}$ where 
\begin{equation}
\chi  (\nabla h + \tilde{W} h) = \mathrm{curl}(\tilde{h}) + \text{error term} \ \text{in}\ B_2,
\end{equation}
see \Cref{lem:approximatestreampolar} below. One can compare this procedure to the one employed in \cite[Section 5]{KSW15} for proving the Landis conjecture in an exterior domain. Now, one can observe that $\gamma = \chi h + i \tilde{h}$ satisfies the following Beltrami equation
\begin{equation}
\label{eq:equationgammaproof}
\partial_{\overline{z}} \gamma  = \alpha \gamma  + \text{error term} \ \text{in}\ B_2,\ \norme{\alpha}_{L^p(B_2)} \leq C \|\tilde{W}\|_{L^{p}(B_{2-c})},
\end{equation}
see \Cref{lem:beltramifirstequation} below. We finally withdraw the zero order term with the use of the Cauchy transform, i.e. defining $\zeta = \exp(-\beta) \gamma$ with $\partial_{\overline{z}} \beta  = \alpha$, we have
\begin{equation}
\label{eq:equationzetaproof}
\partial_{\overline{z}} \zeta = \text{error term} \ \text{in}\ B_2,
\end{equation}
see \Cref{lem:beltramiequationzeta} below. It is worth mentioning that during each step the error term is changing but at the end of this step, it has the following form
\begin{equation}
\text{error term} = \exp(-\beta)\left[\text{local term} + \text{non local term}\right],
\end{equation}
where both local term and non-local term involve the values of $h, \nabla h$ near the images of the disk of the perforation. The particularity of the non-local term is contained in the fact that it is local in the angular variable while it is non-local in the radial variable. This is due to the fact that the approximate stream function has been introduced with respect to polar coordinates in prevision of the next step. We also formulate well-known properties for the Cauchy transform, i.e. $L^{\infty}$ bound on $\beta$ and Hölder's estimate on $\beta$ in function of the $L^p$ bound on $\alpha$, see \Cref{lem:defomegaandT} below.\\

\textit{Differences of Step 3 in comparison to \cite[Act 3]{LMNN20}.} Here, we do not follow \cite[Act 3]{LMNN20} because we have not reduced our equation to a harmonic equation. Our strategy takes its inspiration in \cite[Section 6.1]{LMNN20} that use Carleman estimates for the Laplacian in a perforated domain and \cite[Section 5]{KSW15} that use Carleman estimates for the $\partial_{\overline{z}}$-operator in an exterior domain.

We apply a Carleman estimate to a cut-off version of $\zeta$ near the boundary $\partial B_2$ and near the observation set $B_{r'/2}$, that satisfies a non homogeneous $\partial_{\overline{z}}$ equation. Then, one needs to absorb the source terms, involving local and nonlocal terms depending on the values of $h$, $\nabla h$ near the disks of the perforation. The local term can be absorbed, by using Harnack inequality on $u$ that transfers into Harnack inequality on $h$ at scale $\varepsilon$, precise Hölder's estimates on $\beta$ and the properties of the Carleman weight, see \Cref{lem:estimationlocCarleman} below. The nonlocal term is more difficult to absorb, it can be treated by using the same previous arguments conjugating with the key point that the non-local variable is only radial and the fact that the Carleman weight is a radial function, see \Cref{lem:estimationnonlocCarleman} below. This is why we have introduced the approximate stream function in polar coordinates instead of the more usual Cartesian coordinates.

In order to come back to the original variable, i.e. to obtain an estimate of $|u(x_{\max})|$, we deduce from the Carleman estimate a $W_{\mathrm{loc}}^{1,2}$ estimate of $\zeta$ by using the $H^1$ regularity of the operator $\partial_{\overline{z}}$, then by Sobolev embedding a $L_{\mathrm{loc}}^p$ estimate of $\zeta$ for every $1 \leq p < +\infty$. But this is not sufficient for our purpose. This is why, we add the extra remark telling that one can actually obtain a $L_{\mathrm{loc}}^p$ estimate of $\partial_{\overline{z}}\zeta$ for every $1 \leq p < +\infty$ by using the same strategy leading to the absorption of the local and nonlocal terms, see \Cref{lem:importantcarleman} below. Therefore, we can deduce a $W_{\mathrm{loc}}^{1,p}$ estimate of $\zeta$ for $p>2$ then a $L_{\mathrm{loc}}^{\infty}$ bound on $\zeta$, giving the bound on $|u(x_{\max})|$, leading to the observability estimate \eqref{eq:observabilityu}.

\subsection{Organization of the paper}

In \Cref{sec:sectionmultiplier}, we present the Step 1 of the proof of the main local result \Cref{thm:landislocal}. In \Cref{sec:reductiondzbar}, we present the Step 2 of the proof of the main local result \Cref{thm:landislocal}. In \Cref{sec:carlemanpart}, we present the Step 3 of the proof of the main local result \Cref{thm:landislocal}. Finally, in \Cref{sec:finalcomments}, we address final comments concerning this work.

\bigskip

{\noindent \bf Acknowledgements.} The authors are indebted to Sylvain Ervedoza and Enrique Fern\'andez-Cara for interesting discussions about this work. The authors thank Laboratoire Jacques-Louis Lions, Departamento de Ecuaciones Diferenciales y Análisis Numérico and Instituto de Matemáticas de la  Universidad de Sevilla where part of this work was done. The first author is partially supported by the Project TRECOS ANR-20-CE40-0009 funded by the ANR (2021--2024). The second author is partially supported by Grant PID2020–114976GB–I00, funded by MCIN/AEI/10.13039/501100011033.

\section{Step 1: Construction of a positive multiplier in a suitable perforated domain}
\label{sec:sectionmultiplier}

The main goal of this step is to construct a positive multiplier of the following equation $-\Delta \varphi - \nabla \cdot (W_2  \varphi) + W_1 \cdot \nabla \varphi + V \varphi = 0$. As explained in \Cref{sec:strategyofproof}, such a construction would be made possible by perforating the domain $B_2$ in a suitable way to reduce the Poincaré constant. Indeed, this will allow us to apply weak maximum principles, quantified in function of the Poincaré constant and the parameters of the elliptic operator, to prove the existence of such a function $\varphi$.

\subsection{Weak quantitative maximum principles}

The goal of this first part is to prove maximum principles for elliptic operators in an open bounded set $\Omega$, with a small Poincaré constant. Two kind of results would be first provided: the first one to deal with $L^{\infty}$-source terms, see \Cref{lem:varphiTilde} below, mainly based on de Giorgi's iteration in the spirit of \cite[Lemma 6.10]{LMNN20} and the second one to deal with $W^{-1, \infty}$-source terms, see \Cref{lem:varphiTildeDivergenceSource} below, following standard Stampacchia's iterative strategy. The conjugation of these two results would culminate to a weak quantitative maximum principle for a general elliptic operator, see \Cref{lem:varphidiv} below. The main novelty of these results would be the quantification of the constants in function of the Poincaré constant of $\Omega$ and in function of $L^{\infty}$-bounds of the lower order terms appearing in the elliptic operators.

\subsubsection{With a $L^{\infty}$-source term}

The main result of this part is the following weak maximum principle with a $L^{\infty}$-source term.

\begin{lemma}
	\label{lem:varphiTilde}
	For every $\varepsilon>0$, $C' \geq 1$, there exist $c>0$ and $C>0$, independent of $\varepsilon$, such that for every bounded open set $\Omega \subset \R^2$ with $C_P(\Omega)^2 \leq (C')^2 \varepsilon^2$, $W \in L^{\infty}(\Omega;\R^2)$, $f \in L^{\infty}(\Omega;\R)$, satisfying
	\begin{equation}
	\label{eq:epsPetitPoincareW}
	\varepsilon + \varepsilon \norme{W}_{L^{\infty}(\Omega)}\leq c,
\end{equation}
then there exists a unique $\Phi \in H_0^1(\Omega)$ solution of
	\begin{equation}
		\label{Eq-Phi}
				-\Delta \Phi + W \cdot \nabla \Phi = f     \text{ in } \Omega,
	\end{equation}
satisfying
	\begin{equation}
		\label{Est-Size-Phi}
		 \norme{\Phi}_{\infty} \leq C \varepsilon^2 \norme{f}_{L^{\infty}(\Omega)},
	\end{equation}
	together with
	\begin{equation}
		\label{Est-Size-PhiH0}
		 \norme{\Phi}_{H_0^1(\Omega)} \leq C \varepsilon \norme{f}_{L^{2}(\Omega)}.
	\end{equation}
\end{lemma}
This is a generalization of \cite[Lemma 6.10]{LMNN20} and the new difficulty is the presence of the drift term $W$. In order to prove \Cref{lem:varphiTilde}, we need the following rescaled version.

\begin{lemma}
	\label{lem:varphiTildeRescaled}
There exist $c>0$ small enough and $C>0$ large enough such that for every bounded open set $\Omega$ contained in $\R^2$, with $C_P(\Omega)^2 \leq c^2$, $W \in L^{\infty}(\Omega;\R^2)$, $\norme{W}_{\infty} \leq 1$, $f \in L^{\infty}(\Omega;\R)$, $\norme{f}_{\infty} \leq 1$, there exists a unique $\Phi \in H_0^1(\Omega)$ such that
	\begin{equation}
		\label{Eq-PhiRescaled}
				-\Delta \Phi + W \cdot \nabla \Phi = f     \text{ in } \Omega,
	\end{equation}
	and $\Phi$ satisfies 
	\begin{equation}
		\label{Est-Size-PhiRescaled}
		 \norme{\Phi}_{\infty} \leq C.
	\end{equation}
\end{lemma}
By a scaling argument, we can then deduce the following result.

\begin{proof}[Proof of \Cref{lem:varphiTilde} from \Cref{lem:varphiTildeRescaled}]
Let us set $c_0$ and $C_0$ the constants provided by \Cref{lem:varphiTildeRescaled}. Let us set $$\Omega_0 = \frac{c_0}{C'\varepsilon} \Omega,\ \tilde{\Phi} = \frac{c_0^2}{C'^2 \varepsilon^2 \norme{f}_{L^{\infty}}} \Phi\left(\frac{C'\varepsilon}{c_0}\cdot\right),\ \tilde{W} = \frac{C' \varepsilon}{c_0}W\left(\frac{C'\varepsilon}{c_0}\cdot\right),\  \tilde{f} =   \norme{f}_{L^{\infty}}^{-1} f\left(\frac{C'\varepsilon}{c_0}\cdot\right),$$ 
then $C_P(\Omega_0)^2 \leq c_0^2$, $\|\tilde{W}\|_{\infty} \leq 1$ provided that $c \leq c_0/C'$, $\|\tilde{f}\|_{\infty} \leq 1$ so one can apply \Cref{lem:varphiTildeRescaled} that gives $ \|\tilde{\Phi}\|_{\infty} \leq C_0$, which leads to \eqref{Est-Size-Phi}. For obtaining \eqref{Est-Size-PhiH0}, we test the variational formulation of \eqref{Eq-Phi} with $\Phi$ to get
$$ \int_{\Omega} |\nabla \Phi|^2 + \int_{\Omega} (W \cdot \nabla \Phi) \Phi = \int_{\Omega} f \Phi.$$
We then use Young's inequality, together with Poincaré inequality using the assumption \eqref{eq:epsPetitPoincareW} for obtaining
$$ \int_{\Omega} |\nabla \Phi|^2 \leq C \varepsilon^2 \int_{\Omega} |f|^2,$$
which leads to the desired conclusion.
\end{proof}

The rest of the part is then devoted to the proof of \Cref{lem:varphiTildeRescaled}.

\begin{proof}[Proof of \Cref{lem:varphiTildeRescaled}]

We divide the proof into several steps and $c>0$ is a positive numerical constant that will be fixed later.\\

\textit{Step 1: Existence and uniqueness by Lax-Milgram's lemma.} Set $k^2 = C_P(\Omega)^2 \leq c^2$. Let us introduce
\begin{equation}
a(u,v) = \int_{\Omega} \nabla u \cdot \nabla v + \int_{\Omega}( W \cdot \nabla u) v \qquad \forall u, v \in H_0^1(\Omega).
\end{equation}
It is straightforward to prove that $a$ is a continuous, bilinear form on $H_0^1(\Omega)$. Let us check the coercivity of $a$. For $c<1/2$, by using Young's inequality, 
\begin{equation}
a(u,u) = \int_{\Omega} |\nabla u|^2 + \int_{\Omega}( W \cdot \nabla u) u \geq \frac{1}{2}(1 - k^2) \int_{\Omega} |\nabla u|^2 \geq \frac{3}{8} \int_{\Omega} |\nabla u|^2 = \frac{3}{8} \norme{u}_{H_0^1(\Omega)}^2.
\end{equation}
Let us now consider 
\begin{equation}
l(v) = \int_{\Omega} f v\qquad \forall v \in H_0^1(\Omega).
\end{equation}
It is straightforward to prove that $l$ is a continuous, linear form on $H_0^1(\Omega)$.

Therefore, by Lax-Milgram's lemma, there exists a unique $\Phi \in H_0^1(\Omega)$ such that
\begin{equation}
\label{eq:formulationvariational}
 \int_{\Omega} \nabla \Phi \cdot \nabla v + \int_{\Omega}( W \cdot \nabla \Phi) v = \int_{\Omega} f v\qquad \forall v \in H_0^1(\Omega).
\end{equation}

\textit{Step 2: Local estimate on $\Phi$.}

Now we want to prove some local estimate, i.e. there exists a positive numerical constant $C>0$ such that for every unit ball $B \subset \R^2$, 
\begin{equation}
\label{eq:localestimatePhi}
\int_{B \cap \Omega} |\Phi|^2 \leq C k^4.
\end{equation}
Up to a translation argument, one can assume that $B=B(0,1)$.
Let us introduce 
\begin{equation}
\varphi(x) = \exp(-|x|).
\end{equation}
Then, it is easy to check that $\varphi$ satisfies the following properties
$$ \forall 1 \leq p \leq \infty,\  \varphi \in W^{1,p}(\R^2),\ |\nabla \varphi| \leq \varphi,\ \int_{\R^2} \varphi = 2\pi.$$
Moreover, as a consequence we have that $\varphi \Phi \in H_0^1(\Omega)$. So, one can apply the Poincaré inequality to $\varphi \Phi$, this leads to
\begin{equation}
\int_{\Omega} |\varphi \Phi|^2 \leq k^2 \int_{\Omega} |\nabla (\varphi \Phi)|^2 \leq 2k^2 \int_{\Omega} \varphi^2 |\nabla \Phi|^2 + 2 k^2 \int_{\Omega} \varphi^2 | \Phi|^2,
\end{equation}
hence providing $c < 1/2$, we get
\begin{equation}
\label{eq:firstestimatelocal}
\int_{\Omega} |\Phi|^2 \psi \leq 4 k^2 \int_{\Omega} |\nabla \Phi|^2 \psi,
\end{equation}
where $\psi = \varphi^2$.

Now set $v = \psi \Phi$ that also belongs to $H_0^1(\Omega)$ so one can apply the variational formulation \eqref{eq:formulationvariational} to $v$ to get
\begin{equation}
\label{eq:formvarpsiPhi}
\int_{\Omega}  |\nabla \Phi|^2 \psi + \int_{\Omega} (\nabla \psi \cdot \nabla \Phi) \Phi + \int_{\Omega}( W \cdot \nabla \Phi) \psi \Phi = \int_{\Omega} f \psi \Phi.
\end{equation}
We bound the right hand side of \eqref{eq:formvarpsiPhi} by using \eqref{eq:firstestimatelocal} and $\int_{\R^2} \psi \leq 1$,
\begin{equation}
\label{eq:formvarpsiPhi1}
\left|\int_{\Omega} f \psi \Phi\right| \leq \left(\int_{\Omega}| \Phi|^2 \psi\right)^{1/2} \leq 2 k \left(\int_{\Omega} |\nabla \Phi|^2 \psi\right)^{1/2}.
\end{equation}
For the second term in the left hand side of \eqref{eq:formvarpsiPhi}, we proceed as follows using \eqref{eq:firstestimatelocal}, providing $c<1/16$,
\begin{multline}
\label{eq:formvarpsiPhi2}
\left| \int_{\Omega} (\nabla \psi \cdot \nabla \Phi) \Phi\right| \leq 2  \int_{\Omega} \psi |\nabla \Phi| |\Phi| \leq 2 \left(\int_{\Omega} | \Phi|^2 \psi\right)^{1/2} \left(\int_{\Omega} |\nabla \Phi|^2 \psi\right)^{1/2} \leq 4 k \left(\int_{\Omega} |\nabla \Phi|^2 \psi\right)\\
 \leq \frac{1}{4} \int_{\Omega} |\nabla \Phi|^2 \psi .
\end{multline}
For the third left hand side term of \eqref{eq:formvarpsiPhi}, we proceed as follows using \eqref{eq:firstestimatelocal} and $\|W\|_{\infty} \leq 1$, providing $c<1/8$,
\begin{multline}
\label{eq:formvarpsiPhi3}
\left| \int_{\Omega} (W \cdot \nabla \Phi) \psi \Phi\right| \leq  \int_{\Omega} \psi |\nabla \Phi| |\Phi| \leq \left(\int_{\Omega} |\Phi|^2 \psi\right)^{1/2} \left(\int_{\Omega} |\nabla \Phi|^2 \psi\right)^{1/2} \leq 2 k \left(\int_{\Omega} |\nabla \Phi|^2 \psi\right)\\
 \leq \frac{1}{4} \int_{\Omega}  |\nabla \Phi|^2 \psi.
\end{multline}
By conjugating \eqref{eq:formvarpsiPhi}, \eqref{eq:formvarpsiPhi1}, \eqref{eq:formvarpsiPhi2} and \eqref{eq:formvarpsiPhi3} we get for $c < 1/16$,
\begin{equation}
\int_{\Omega}  |\nabla \Phi|^2 \psi \leq 4 k \left(\int_{\Omega} |\nabla \Phi|^2 \psi \right)^{1/2},
\end{equation}
so 
\begin{equation}
\label{eq:secondestimatelocal}
\int_{\Omega}  |\nabla \Phi|^2 \psi \leq 16 k^2.
\end{equation}
By using \eqref{eq:firstestimatelocal} and \eqref{eq:secondestimatelocal}, we get the expected result \eqref{eq:localestimatePhi} with $C=64$.\\

\textit{Third step: Poincaré constant of thin domains.} We have the following result, that is exactly \cite[Corollary 6.9]{LMNN20}.
\begin{lemma}
\label{lemma:poincarecarre}
There exists $c_0>0$ small enough such that for every $k>0$, for every bounded open set $\Omega \subset \R^2$ satisfying
\begin{equation}
| \Omega \cap Q | \leq k^2 \leq c_0^2\ \qquad \text{for any square}\ Q\ \text{with $1/2$ side-length},
\end{equation}
then $C_P(\Omega)^2 \leq C k^2$ for some numerical constant $C>0$, independent of $k$.
\end{lemma}

\textit{Step 4: De Giorgi scheme.}

We now fix $c= \min (1/32,c_0)>0$ where $c_0>0$ is the constant in \Cref{lemma:poincarecarre}. Let $t_0>0$ that we will be fixed later and $\Omega_0 = \{ \Phi > t_0\} \subset \Omega$ with $k_0^2 = C_P(\Omega_0)^2$. From \eqref{eq:localestimatePhi}, we get
\begin{equation}
\label{eq:localestimatePhi0}
\int_{B \cap \Omega} |\Phi|^2 \leq C k^4,
\end{equation}
then
\begin{equation}
| \{ \Phi > t_0 \} \cap B | \leq \frac{C k^4}{t_0^2}.
\end{equation}
So, by using \Cref{lemma:poincarecarre},
$$k_0^2 \leq \frac{C k^4}{t_0^2}.$$
Then, let us set $t_0 =  \sqrt{C k}$ leading to $k_0^2 \leq k^3 \leq c^3$.

We now recall the well-known facts: $H_0^1(\Omega_0) \subset H_0^1(\Omega)$ and $\Phi_0 := (\Phi-t_0)^{+} \in H_0^1(\Omega_0)$ with $\nabla \Phi_0 = \nabla \Phi 1_{\Omega_0}$, see for instance \cite[Proposition 1.3.10]{WYW06}. Applying the variational formulation \eqref{eq:formulationvariational} we then get
\begin{equation}
\label{eq:formulationvariationalOmega0}
 \int_{\Omega_0} \nabla \Phi_0 \cdot \nabla v + \int_{\Omega_0}( W \cdot \nabla \Phi_0) v = \int_{\Omega_0} f v\qquad \forall v \in H_0^1(\Omega_0).
\end{equation}
We then iterate the previous arguments, that is we first prove the local estimate on $\Phi_0$, there exists a positive numerical constant $C>0$ such that for every unit ball $B \subset \R^2$, 
\begin{equation}
\label{eq:localestimatePhi0}
\int_{B \cap \Omega_0} |\Phi_0|^2 \leq C k_0^4.
\end{equation}
Let $t_1>0$ that we will be fixed later and $\Omega_1 = \{ \Phi_0 > t_1\} = \{ (\Phi-t_0)^+ > t_1\} \subset \Omega_0$,  $k_1^2 = C_P(\Omega_1)^2$. We then obtain from \eqref{eq:localestimatePhi0} for every unit ball $B \subset \R^2$,
\begin{equation}
| \{ \Phi_0 > t_1 \} \cap B | \leq \frac{C k_0^4}{t_1^2}.
\end{equation}
So, by using \Cref{lemma:poincarecarre},
$$k_1^2  \leq \frac{C k_0^4}{t_1^2}.$$
Then, let us set $t_1 =  \sqrt{C k_0}$ leading to $k_1^2 \leq k_0^3$.

By induction, we can construct $$t_n = \sqrt{C k_{n-1}},\ \Omega_n = \{\Phi_{n-1} > t_n\},\ k_n^2 = C_P(\Omega_n)^2,\ \Phi_n = (\Phi_{n-1}-t_n)^{+}\qquad \forall n \in \N,$$
with the convention $k_{-1} = k = C_p(\Omega)$, $\Phi_{-1} = \Phi$, leading to
$$ k_{n+1} \leq \left(c^{3/2}\right)^{n+2}\qquad \forall n \geq 0.$$

With such a construction, we have because $c \leq 1/2$,
\begin{align}
\sum_{n=0}^{+\infty} t_n &\leq \sqrt{C} \sum_{n=-1}^{+\infty} 2^{-\frac{3(n+1)}{2}}:=T,\label{eq:estimationT}\\
| \{ \Phi_n > t_{n+1} \} \cap B |& \leq k_{n}^3\qquad \text{for every unit ball}\  B \subset \R^2,\ \forall n \in \N,\label{eq:localestimatePhin}\\
\Phi& \leq t_0 + t_1 + \dots + t_n + \Phi_{n}\qquad \forall n \in \N.\label{eq:phifunctionphin}
\end{align}
Therefore, for every unit ball $B \subset \R^2$, we have from \eqref{eq:localestimatePhin} that
$$ |\{\Phi_n > t_{n+1}\} \cap B| \to 0\ \text{as}\ n \to +\infty,$$
hence conjugating with \eqref{eq:estimationT} and \eqref{eq:phifunctionphin},
\begin{multline}
| \{\Phi > 2 T\} \cap B | \leq | \{\Phi > 2 T\} \cap \{\Phi_n \leq t_{n+1}\} \cap B| + | \{\Phi > 2 T\} \cap \{\Phi_n > t_{n+1}\} \cap B|\\ \leq | \{\Phi_n > t_{n+1}\} \cap B| \to 0\ \text{as}\ n \to +\infty.
\end{multline}
Then $| \{\Phi > 2 T\}| = 0$ so $\Phi \leq 2 T$ almost everywhere

By linearity, using that $-\Phi$ solves \eqref{Eq-PhiRescaled} replacing $f$ by $-f$, we then obtain with the same strategy that $-\Phi \leq 2 T$ then the expected bound \eqref{Est-Size-PhiRescaled}.
\end{proof}

\subsubsection{With a $W^{-1, \infty}$-source term}

The main result of this part is the following weak maximum principle with a $W^{-1,\infty}$-source term.

\begin{lemma}
	\label{lem:varphiTildeDivergenceSource}
	For every $\varepsilon>0$, $C' \geq 1$, $p>2$, there exist $c>0$ and $C>0$, independent of $\varepsilon$, such that for every bounded open set $\Omega \subset \R^2$ with $C_P(\Omega)^2 \leq (C')^2 \varepsilon^2$, $W \in L^{\infty}(\Omega;\R^2)$, $f \in L^{\infty}(\Omega;\R)$, satisfying
	\begin{equation}
	\label{eq:epsPetitPoincareWDiv}
	\varepsilon + \varepsilon \norme{W}_{L^{\infty}(\Omega)}\leq c,
\end{equation}
there exists a unique $\Phi \in H_0^1(\Omega)$ such that
	\begin{equation}
		\label{Eq-PhiDiv}
				-\Delta \Phi + W \cdot \nabla \Phi = \nabla \cdot g     \text{ in } \Omega,
	\end{equation}
	and $\Phi$ satisfies 
	\begin{equation}
		\label{Est-Size-PhiDiv}
		 \norme{\Phi}_{\infty} \leq C |\Omega|^{(p-2)/2p} \varepsilon^{2/p} \norme{g}_{L^{\infty}(\Omega)},
	\end{equation}
	together with
	\begin{equation}
		\label{Est-Size-PhiH0Div}
		 \norme{\Phi}_{H_0^1(\Omega)} \leq C  \norme{g}_{L^{2}(\Omega)}.
	\end{equation}
\end{lemma}

The main difference between \Cref{lem:varphiTildeDivergenceSource} and \Cref{lem:varphiTilde} is that the bound on \eqref{Est-Size-PhiDiv} depends explicitly on the measure of $\Omega$. This is due to the fact that in the following proof, we will argue differently. Indeed, we will follow standard Stampacchia's iterative strategy on the measure of the level sets of $u$, instead of the previous iterative strategy on the Poincaré constant on the level sets of $u$. We do not know if one can remove the dependence on the measure of $\Omega$ and more importantly replace $\varepsilon^{2/p}$ in \eqref{Est-Size-PhiDiv} by $\varepsilon$.

To prove \Cref{lem:varphiTildeDivergenceSource}, we will use the following precise Sobolev's inequality.
\begin{lemma}
\label{lem:sobolevinequalities}
For every $\varepsilon>0$, $C' \geq 1$, $p\geq 2$, there exists $C>0$, independent of $\varepsilon$, such that for every bounded open set $\Omega \subset \R^2$ with $C_P(\Omega)^2 \leq (C')^2 \varepsilon^2$, we have
\begin{equation}
\norme{u}_{L^p(\Omega)} \leq C \varepsilon^{2/p} \norme{\nabla u}_{L^2(\Omega)}\ \qquad \forall u \in H_0^1(\Omega).
\end{equation}
\end{lemma}
\begin{proof}
Let $u \in C_c^{1}(\Omega)$. We extend $u$ by $0$ outside $\Omega$, then we get that $u \in C_c^{1}(\R^2)$. Therefore, one can apply \cite[Equation (20) Page 280]{Bre11} to get
\begin{equation}
\norme{u}_{2m}^m \leq m \norme{u}_{2(m-1)}^{m-1} \norme{\nabla u}_{2}\qquad \forall m \geq 1.
\end{equation}
Take $m=2$, we obtain
\begin{equation}
\norme{u}_{4}^2 \leq 2 \norme{u}_{2} \norme{\nabla u}_{2}.
\end{equation}
Then, we apply Poincaré's inequality, we get
\begin{equation}
\norme{u}_{4}^2 \leq C \varepsilon \norme{\nabla u}_{2}^2.
\end{equation}
Therefore, we obtain the conclusion of the lemma with $p=4$. Then, we iterate and the conclusion follows by interpolation.
\end{proof}

Let us now dedicate the end of this part to the proof of \Cref{lem:varphiTildeDivergenceSource}.

\begin{proof}
The existence and uniqueness of such a $\Phi$ comes from a classical application of the Lax-Milgram lemma.

Set $\Phi_k = (\Phi - k)^{+}$ where $k >0$, $A(k) = \{x \in \Omega\ ;\ \Phi > k\}$. We have that $\Phi_k \in H_0^1(\Omega)$, then we have
\begin{equation}
\int_{\Omega} |\nabla \Phi_k|^2 + \int_{\Omega} (W \cdot \nabla \Phi_k) \Phi_k \leq \left(\int_{A(k)} |g|^2\right)^{1/2}\left(\int_{\Omega} |\nabla \Phi_k|^2\right)^{1/2}.
\end{equation}
The second left hand side term can be absorbed by using Poincaré inequality and \eqref{eq:epsPetitPoincareWDiv} as follows
\begin{equation}
\left|\int_{\Omega} (W \cdot \nabla \Phi_k) \Phi_k\right| \leq C \varepsilon \|W\|_{\infty} \int_{\Omega} |\nabla \Phi_k|^2 \leq c  \int_{\Omega} |\nabla \Phi_k|^2.
\end{equation}
Therefore, for $p>2$, by using the Sobolev's inequality, we have
\begin{equation}
\varepsilon^{-2/p} \norme{\Phi_k}_p \leq C \left(\int_{\Omega} |\nabla \Phi_k|^2\right)^{1/2} \leq  C\left(\int_{A(k)} |g|^2\right)^{1/2} \leq C|A(k)|^{1/2} \norme{g}_{\infty}.
\end{equation}
So, for $h > k$, using that $A(h) \subset A(k)$, and on $A(h)$, we have $\Phi_k \geq h-k$, we then have
\begin{equation}
|A(h)|^{1/p} (h-k) \leq C \varepsilon^{2/p} |A(k)|^{1/2} \norme{g}_{\infty},
\end{equation}
so
\begin{equation}
|A(h)| \leq \left(\frac{C \varepsilon^{2/p}\norme{g}_{\infty}}{h-k}\right)^p |A(k)|^{p/2}.
\end{equation}
By using that $p/2>1$, we have from \cite[Lemma 4.1]{Sta65},
\begin{equation}
|A(k)| = 0\qquad \forall k \geq C  \varepsilon^{2/p}\norme{g}_{\infty} |\Omega|^{(p/2-1 )/p}.
\end{equation}
Then, 
\begin{equation}
\Phi \leq C |\Omega|^{(p-2)/2p}\varepsilon^{2/p}\norme{g}_{\infty}.
\end{equation}
The same arguments also give
\begin{equation}
\Phi \geq -C |\Omega|^{(p-2)/2p} \varepsilon^{2/p}\norme{g}_{\infty},
\end{equation}
so the conclusion \eqref{Est-Size-PhiDiv}. 

The estimate \eqref{Est-Size-PhiH0Div} directly comes from the application of the variational formulation to $\Phi$.
\end{proof}

By using \Cref{lem:varphiTilde} and \Cref{lem:varphiTildeDivergenceSource}, we can now obtain the following result that is the main result of this part. 
\begin{prop}
	\label{lem:varphidiv}
	For every $\varepsilon>0$, $C' \geq 1$, $p>2$, there exist $c>0$ and $C>0$, independent of $\varepsilon$, such that for every bounded open set $\Omega \subset \R^2$ with $C_P(\Omega)^2 \leq (C')^2 \varepsilon^2$, $W_1,W_2 \in L^{\infty}(\Omega;\R^2)$, $V \in L^{\infty}(\Omega;\R)$, satisfying
\begin{equation}
	\label{eq:epsPetitPoincareGeneralpreuve}
	\varepsilon + \varepsilon^{2/p} \norme{W_1}_{L^{\infty}(\Omega)} + \varepsilon\norme{W_2}_{L^{\infty}(\Omega)} + \varepsilon^2 \norme{V}_{L^{\infty}(\Omega)} \leq c,
\end{equation}
there exists a unique $\varphi \in H^1(\Omega)$ such that
	\begin{equation}
		\label{Eq-varphiNodal}
				-\Delta \varphi - \nabla \cdot( W_1 \varphi) + W_2 \cdot \nabla \varphi +  V \varphi = 0  \text{ in } \Omega, 	
	\end{equation}
	and $ \tilde{\varphi} = \varphi - 1$ satisfies 
	\begin{equation}
		\label{Est-Size-tilde-varphiNodal}
		\tilde{\varphi} \in H_0^1(\Omega)\text{ and }  \norme{\tilde{\varphi}}_{\infty} \leq C  \left( \varepsilon^{2/p} |\Omega|^{(p-2)/2p} \norme{W_1}_{L^{\infty}(\Omega)} + \varepsilon^2 \norme{V}_{L^{\infty}(\Omega)}\right).
	\end{equation}
\end{prop}
This result has to be compared with \cite[Lemma 3.2]{LMNN20}. The new difficulty comes from the presence of the drift terms $W_1, W_2$. 
\begin{proof}[Proof of \Cref{lem:varphidiv} from \Cref{lem:varphiTilde} and \Cref{lem:varphiTildeDivergenceSource}]
By \Cref{lem:varphiTilde} and \Cref{lem:varphiTildeDivergenceSource}, let $\Phi_0 \in H_0^1(\Omega)$ be the unique solution satisfying
\begin{equation}
		\label{Eq-Phi0}
				-\Delta \Phi_0 + W_2 \cdot \nabla \Phi_0 = - V + \nabla \cdot(W_1)   \text{ in } \Omega,
	\end{equation}
From \eqref{Est-Size-Phi} and \eqref{Est-Size-PhiDiv}, we have 
$ \norme{\Phi_0}_{\infty} \leq C \left(\varepsilon^{2/p} |\Omega|^{(p-2)/2p} \norme{W_1}_{\infty} + \varepsilon^2 \norme{V}_{\infty} \right).$ From \eqref{Est-Size-PhiH0} and \eqref{Est-Size-PhiH0Div}, we also have $ \norme{\Phi_0}_{H_0^1(\Omega)} \leq C (|\Omega|^{1/2} \norme{W_1}_{\infty} + \varepsilon |\Omega|^{1/2}  \norme{V}_{\infty} ) .$
For $n \geq 1$, we set $\Phi_n \in H_0^1(\Omega)$ the unique solution satisfying
\begin{equation}
		\label{Eq-Phin}
				-\Delta \Phi_n + W_2 \cdot \nabla \Phi_n = -V \Phi_{n-1}  + \nabla \cdot ( W_1 \Phi_{n-1})   \text{ in } \Omega,
	\end{equation}
We then have from \eqref{Est-Size-Phi} and \eqref{Est-Size-PhiDiv},
$ \norme{\Phi_n}_{\infty} \leq C \left(\varepsilon^{2/p} |\Omega|^{(p-2)/2p} \norme{W_1}_{\infty} + \varepsilon^2 \norme{V}_{\infty}\right) \norme{\Phi_{n-1}}_{\infty}$ for every $n \geq 1$. Therefore, assuming that $C \left(\varepsilon^{2/p} |\Omega|^{(p-2)/2p} \norme{W_1}_{\infty} + \varepsilon^2 \norme{V}_{\infty} \right)\leq 1/2$, we have $$\norme{\Phi_n}_{\infty} \leq C\left(\varepsilon^{2/p} |\Omega|^{(p-2)/2p}\norme{W_1}_{\infty} + \varepsilon^2 \norme{V}_{\infty}\right) 2^{-n}\ \forall n \geq 0.$$ 
Moreover, we also have $\norme{\Phi_n}_{H_0^1(\Omega)} \leq C(  |\Omega|^{1/2} \norme{W_1}_{\infty}  + \varepsilon |\Omega|^{1/2} \norme{V}_{\infty} )\norme{\Phi_{n-1}}_{\infty} $ for $n \geq 1$, leading to $$\norme{\Phi_n}_{H_0^1(\Omega)} \leq C |\Omega|^{1/2} (\norme{W_1}_{\infty} + \norme{V}_{\infty} )  2^{-n}\qquad \forall n \geq 0.$$
 Therefore, the series $\sum_{n \geq 0} \Phi_n$ converges absolutely then converges to $\tilde{\varphi}$ in $L^{\infty}$ and in $H_0^1(\Omega)$. Moreover, we have
$$ \norme{\tilde{\varphi}}_{\infty} \leq \sum_{n \geq 0} \norme{\Phi_n}_{\infty} \leq C \left(\varepsilon^{2/p} |\Omega|^{(p-2)/2p} \norme{W_1}_{\infty} + \varepsilon^2 \norme{V}_{\infty}\right).$$
Furthermore, we have that 
\begin{equation*}
				-\Delta \tilde{\varphi}  - \nabla \cdot (W_1 \tilde{\varphi}) + W_2 \cdot \nabla \tilde{\varphi} + V \tilde{\varphi} = - V + \nabla \cdot(W_1) \text{ in } \Omega, 
	\end{equation*}
Hence, $\varphi = \tilde{\varphi} + 1$ satisfies \eqref{Eq-varphiNodal}. This concludes the proof. 
\end{proof}

\subsection{Properties of the nodal set and perforation process}

\textbf{Setting of parameters.} Fix now 
\begin{equation}
\label{eq:deffirstp}
p = 2+\delta > 2,
\end{equation}
where $\delta>0$ is as in \Cref{thm:landislocal}. Take $\varepsilon>0$ a free parameter satisfying 
\begin{equation}
	\label{eq:epsPetitPoincareGeneral1}
	\varepsilon + \varepsilon^{2/(2+\delta)} \norme{W_1}_{L^{\infty}(B_2)} + \varepsilon \norme{W_2}_{L^{\infty}(B_2)} + \varepsilon^2 \norme{V}_{L^{\infty}(B_2)} \leq c.
\end{equation}

Let us now give an application of \Cref{lem:varphidiv} to establish the fundamental property on the nodal set of $u$, that we called before \eqref{eq:Pepsilon}.
\begin{lemma}
\label{lem:nodalsetu}
Let $u$ be a real-valued solution to $-\Delta u - \nabla \cdot (W_1  u) + W_2 \cdot \nabla u + Vu =0$ in a ball $B(x,\varepsilon)$ with $\varepsilon>0$ satisfying \eqref{eq:epsPetitPoincareGeneral1} and $u \in H^1(B(x,\varepsilon)) \cap C^0(\overline{B(x,\varepsilon)})$. Then, if $u>0$ on $\partial B(x, \varepsilon)$ then $u >0$ in $B(x, \varepsilon)$.
\end{lemma}
\begin{proof}
We may assume that $u \geq \alpha >0$ on $\partial B(x, \varepsilon)$ by continuity. We argue by contradiction, assume there exists $x_0 \in B(x,\varepsilon)$ such that $u(x_0) \leq 0$.

Consider now the set $\Omega = \{y \in B(x, \varepsilon)\ ;\ u(y) < \alpha /2\}$. This is an open set strictly inside $B(x,\varepsilon)$ because $u \geq \alpha$ on $\partial B(x,\varepsilon)$ and it is not empty because $u(x_0) \leq 0< \alpha/2$.

Since $u \in C^0(\overline{\Omega})$ and $u = \alpha /2$ on $\partial\Omega$ by continuity, we then know that $(u-\alpha/2)$ belongs to $H_0^1(\Omega)$, see for instance \cite[Theorem 9.17, (i) $\Rightarrow$ (ii)]{Bre11}. Remark that $\Omega \subset B(x,\varepsilon)$ so $C_P(\Omega)^2 \leq C\varepsilon^2$. Then, by \Cref{lem:varphidiv} recalling that $\varepsilon$ satisfies \eqref{eq:epsPetitPoincareGeneral1}, one can construct $\varphi \in H^1(\Omega)$ such that $3/4 \leq \varphi \leq 5/4$ because of \eqref{Est-Size-tilde-varphiNodal} and $-\Delta \varphi - \nabla \cdot(W_1 \varphi) + W_2 \cdot \nabla \varphi + V\varphi =0$ in $\Omega$ and $\varphi -1 \in H_0^1(\Omega)$. Therefore, $((\alpha/2) \varphi - \alpha/2)=(\alpha/2)(\varphi-1)$ belongs to $H_0^1(\Omega)$ and hence the function $g = ((\alpha/2) \varphi - u)$ belongs to $H_0^1(\Omega)$. Moreover, the function $g$ satisfies $-\Delta g - \nabla  \cdot (W_1  g) + W_2 \cdot \nabla g + Vg =0$ in $\Omega$ by linearity. By testing the associated variational formulation with $g$, using Young's inequality, we get for $\varepsilon$ satisfying \eqref{eq:epsPetitPoincareGeneral1} that $\int_{\Omega} |\nabla g|^2 \leq 0$ so $g=0$ in $\Omega$ because $g \in H_0^1(\Omega)$. Therefore, $u = (\alpha/2) \varphi$ in $\Omega$ but $\varphi > 1/2$ in $\Omega$. So $u > \alpha/4$ in $\Omega$, then $u > \alpha / 4$ in $B(x, \varepsilon)$, leading to a contradiction.
\end{proof}

\begin{cor}
\label{cor:nodalset}
Let $u$ be as in \Cref{thm:landislocal}. Then, the nodal set of $u$, 
\begin{equation}
Z := \{x \in B(0,2)\ ;\ u(x) = 0\},
\end{equation}
is closed in $B(0,2)$ and satisfies the following property
\begin{equation}
\label{eq:PepsilonBis}
\forall x_0 \in Z,\ \forall \rho \in (0,\varepsilon),\ \partial B(x_0,\rho) \cap (Z \cup \partial B(0,2)) \neq \emptyset.
\tag{P-$\varepsilon$}
\end{equation}
\end{cor}
\begin{proof}
Let $u$ be as in \Cref{thm:landislocal}, then $u \in W_{\mathrm{loc}}^{1,q}(B_2)$ for every $1 \leq q < +\infty$ by elliptic regularity, then by Sobolev embedding $u \in C^0(B_2)$. This immediately gives that $Z$ is closed in $B(0,2)$. Moreover, the property \eqref{eq:PepsilonBis} is a direct application of \Cref{lem:nodalsetu}.
\end{proof}

Let us take $x_{\max} \in \overline{B_{1}}$ such that
\begin{equation}
	\label{Def-x-max}
	|u(x_{\max})| = \sup_{\overline{B_{1}}} |u|.
\end{equation}

The next step is to construct a suitable perforation of the domain $B_2$ which avoids the nodal set $Z$, $\partial B(0,2)$, $x_{\max}$ and $0$.

%

From \Cref{cor:nodalset}, we then get the following lemma, that is stated in \cite[Section 3.1]{LMNN20} (see also \cite[Lemma 2.10]{ELB22}).

\begin{lemma}
\label{lem:constructionperforated}
	For all $C_0 \geq 5$, for every $\varepsilon >0$, there exist finitely many $C_0\varepsilon$-separated closed disks of radius $\varepsilon$, whose union is denoted by $F_{\varepsilon}$, satisfying the following properties:
	
$\bullet$ these disks are $C_0\varepsilon$-separated from each other, from $Z$, from $\partial B(0,2)$, from $x_{\max}$ and from $0$,

$\bullet$ the set $Z \cup F_{\varepsilon} \cup {\partial B(0,2)}$ is a $6 C_0\varepsilon$-net in $B(0,2)$, meaning that for all $x \in B(0,2)$, $B(x, 6 C_0 \varepsilon) \cap ( Z \cup F_{\varepsilon} \cup {\partial B(0,2)}) \neq \emptyset$.

$\bullet$ the set 
	\begin{equation}
		\label{Omega-Eps}
		\Omega_{\varepsilon} := B(0,2) \setminus (Z \cup F_{\varepsilon})
	\end{equation}
	satisfies $C_P(\Omega_{ \varepsilon} )^2 \leq C^2 \varepsilon^2$ for some constant $C>0$ depending on $C_0$ but independent of $\varepsilon$, $u$, $W_1$, $W_2$ and $V$.
\end{lemma}
%
%

\textbf{Setting of parameters.} In the sequel, it will be useful to choose a very large $C_0$. For simplicity, from now on, we set $C_0 = 18 \cdot 32^2 $. This choice will be made clearer later. \\

In Figure \ref{fig:Perforation}, we have represented the perforated domain. Note that from \cite{HS89}, assuming that $W_1 \in W^{1, \infty}(B_2)$, the nodal set of $u$ is a union of smooth curves. The picture illustrates in particular this structural result, but it is worth mentioning that in this paper we will not use it.

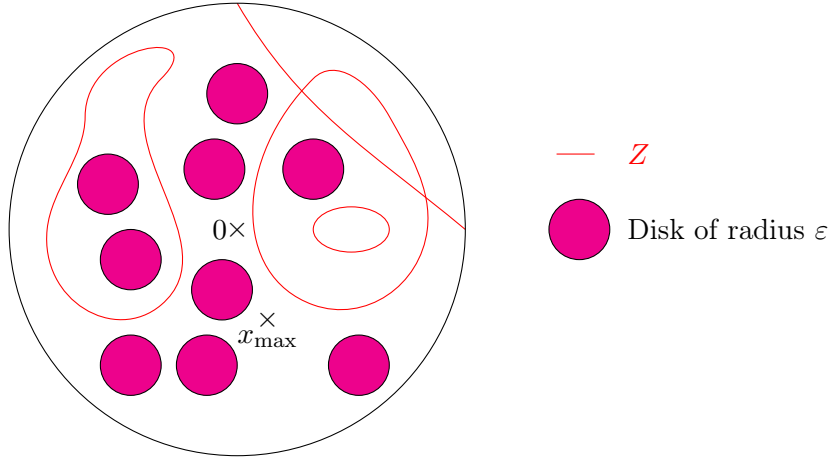
\begin{figure}
\centering
\begin{tikzpicture}
\coordinate (N) at (-1,2);
\coordinate (O) at (-2,1.5);
\coordinate (P) at (-2.5,0);
\coordinate (Q) at (-1,-1);
\coordinate (R) at (-1.5,0);
\coordinate (S) at (-1.6,0.1);
\coordinate (T) at (-1.4,0.2);
\coordinate (U) at (-1.2,0.3);
\coordinate (Na) at (1,2);
\coordinate (Oa) at (2,1.5);
\coordinate (Pa) at (2.5,0);
\coordinate (Qa) at (1,-1);
\coordinate (Nb) at (3.0,0);
\coordinate (Ob) at (2.0,1);
\coordinate (Pb) at (2.0,2);
\coordinate (Qb) at (0,3);
 \draw[red] (N) to [closed, curve through = {(N) (O)  (P)  (Q)}] (N);
 \draw[red] (Na) to [closed, curve through = {(Na) (Oa)  (Pa)  (Qa)}] (Na);
\draw  [red]  (Qb) to[out=-60,in=140] (Nb);
 \draw[red] (1.5,0) ellipse (0.5 and 0.3);
\draw (0,0) circle (3);
\draw[fill=magenta] (0,1.8) circle (0.4);
\draw[fill=magenta] (-0.3,0.8) circle (0.4);
\draw[fill=magenta] (1.0,0.8) circle (0.4);
\draw[fill=magenta] (-0.2,-0.8) circle (0.4);
\draw[fill=magenta] (-0.4,-1.8) circle (0.4);
\draw[fill=magenta] (-1.4,-1.8) circle (0.4);
\draw[fill=magenta] (-1.4,-0.4) circle (0.4);
\draw[fill=magenta] (-1.7,0.6) circle (0.4);
\draw[fill=magenta] (1.6,-1.8) circle (0.4);
\draw (0,0) node[left]{$0$};
\draw (0,0) node {$\times$};
\draw (0.4,-1.2) node [below]{$x_{\max}$};
\draw (0.4,-1.2) node {$\times$};
\draw[fill=magenta] (4.5,0) circle (0.4);
\draw (5.0,0) node[right]{Disk of radius $\varepsilon$};
\draw[red] (4.2,1.0) -- (4.7,1.0);
\draw[red] (5.0,1.0) node[right]{$Z$};
\end{tikzpicture}
\caption{The perforation process with $\Omega_{\varepsilon} = B(0,2) \setminus (\textcolor{red}{Z} \cup \textcolor{magenta}{F_{\varepsilon}})$.}
\label{fig:Perforation}
\end{figure}

\subsection{Construction of the positive multiplier}

\textbf{Setting of parameters.} Note that now, $p$ is as in \eqref{eq:deffirstp} and $\varepsilon>0$ is still a free parameter satisfying
\begin{equation}
	\label{eq:epsPetitPoincareGeneralFinStep1}
	\varepsilon + \varepsilon^{2/(2+\delta)} \norme{W_1}_{L^{\infty}(B_2)} + \varepsilon^{2/(2+\delta)} \norme{W_2}_{L^{\infty}(B_2)} + \varepsilon^2 \norme{V}_{L^{\infty}(B_2)} \leq c,
\end{equation} 
where $c>0$ is a small positive constant depending on the constant $C$ that appears in \Cref{lem:constructionperforated}.\\

We have the following result, that is the main result of this Step 1.
\begin{prop}
\label{prop:constructionmultiplier}
Let $\Omega_{\varepsilon}$ be as in \Cref{lem:constructionperforated}. There exists $\varphi \in H^1(\Omega_{\varepsilon})$ such that
	\begin{equation}
		\label{Eq-varphiOeps}
				-\Delta \varphi - \nabla \cdot (W_2 \varphi) +  W_1 \cdot \nabla \varphi +   V \varphi = 0  \text{ in } \Omega_{\varepsilon}, 
	\end{equation}
	and $ \tilde{\varphi} = \varphi - 1$ satisfies 
	\begin{equation}
		\label{Est-Size-tilde-varphiOeps}
		\tilde{\varphi} \in H_0^1(\Omega_{\varepsilon})\text{ and }  \norme{\tilde{\varphi}}_{\infty} \leq C  \left( \varepsilon^{2/(2+\delta)} \norme{W_2}_{L^{\infty}(B_2)} + \varepsilon^2 \norme{V}_{L^{\infty}(B_2)}\right).
	\end{equation}
\end{prop}
\begin{proof}
It is a direct application of \Cref{lem:varphidiv} to $\Omega = \Omega_{\varepsilon}$, reversing the role of $W_1$ and $W_2$, using \eqref{eq:epsPetitPoincareGeneralFinStep1}.
\end{proof}

\section{Step 2: Reduction to a non-homogeneous $\partial_{\overline{z}}$-equation}
\label{sec:reductiondzbar}

The goal of this step is to use the multiplier $\varphi$, defined in $\Omega_{\varepsilon}$ of the previous Step, introduced in \Cref{prop:constructionmultiplier} to transform first the equation \eqref{eq:equationuLandisquantitativelocal} in a divergence elliptic equation in a subset of $B_2$. Then, by using a quasiconformal change of variable, we will recast this divergence elliptic equation into an elliptic equation of the form $-\Delta h - \nabla \cdot(\tilde{W} h) = 0$. Finally, by an approximate type Poincaré lemma and Cauchy transformation, we will be able to simplify the last equation to a non-homogeneous $\partial_{\overline{z}}$-equation.

\subsection{The new equation satisfied by $v = u/\varphi$}

The first step is to rewrite the elliptic problem $-\Delta u - \nabla \cdot ( W_1 u) + W_2 \cdot \nabla u + V u = 0$ in $B_2$ into an equation of divergence form. 

Unfortunately, we are not able to do it in the whole set $B_2$ directly, but only in the set 
\begin{equation}
	\label{Def-Om-Eps-'}
	\Omega_\varepsilon' = B_2 \setminus F_\varepsilon,  
\end{equation}
i.e. a set which is slightly larger than the set $\Omega_\varepsilon=  B_2\setminus (Z \cup F_{\varepsilon})$ defined in \eqref{Omega-Eps}.

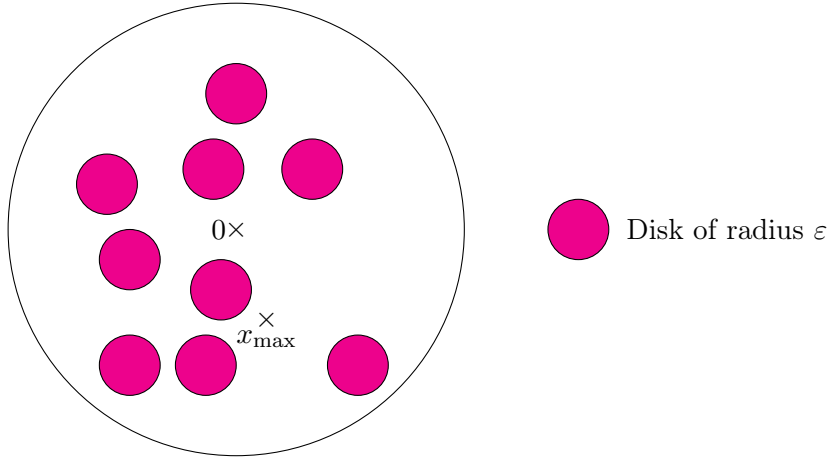
\begin{figure}
\centering
\begin{tikzpicture}
\draw (0,0) circle (3);
\draw[fill=magenta] (0,1.8) circle (0.4);
\draw[fill=magenta] (-0.3,0.8) circle (0.4);
\draw[fill=magenta] (1.0,0.8) circle (0.4);
\draw[fill=magenta] (-0.2,-0.8) circle (0.4);
\draw[fill=magenta] (-0.4,-1.8) circle (0.4);
\draw[fill=magenta] (-1.4,-1.8) circle (0.4);
\draw[fill=magenta] (-1.4,-0.4) circle (0.4);
\draw[fill=magenta] (-1.7,0.6) circle (0.4);
\draw[fill=magenta] (1.6,-1.8) circle (0.4);
\draw (0,0) node[left]{$0$};
\draw (0,0) node {$\times$};
\draw (0.4,-1.2) node [below]{$x_{\max}$};
\draw (0.4,-1.2) node {$\times$};
\draw[fill=magenta] (4.5,0) circle (0.4);
\draw (5.0,0) node[right]{Disk of radius $\varepsilon$};
\end{tikzpicture}
\caption{The set $\Omega_{\varepsilon}' = B(0,2) \setminus   \textcolor{magenta}{F_{\varepsilon}}$.}
\label{fig:PerforationBis}
\end{figure}

Using the equation of $\varphi$ in \eqref{Eq-varphiOeps}, it is clear that, setting $v = u / \varphi$ in $\Omega_\varepsilon$, we have $-\nabla \cdot( \varphi^2 (\nabla v + \hat{W} v)) = 0$ in $\Omega_\varepsilon$ with $\hat{W} = W_1 - W_2$. Extend $\varphi$ by $1$ to $B_2$. In fact, since $\Omega_\varepsilon' = \Omega_\varepsilon \cup Z$, and $u$ vanishes on $Z$, an adaptation of \cite[Lemma 4.1]{LMNN20} yields that the equation $-\nabla \cdot( \varphi^2 (\nabla v + \hat{W} v)) = 0$ also holds in $\Omega_\varepsilon'$. To be more precise, we get the following result.

\begin{lemma}
	\label{lem:divergenceperforated}
	The function $v$ defined in $\Omega_\varepsilon'$ by
	\begin{equation}
		\label{Def-v}
		v := \frac{u}{\varphi} \text{ in } \Omega_\varepsilon',
	\end{equation}
	belongs to $H^1(\Omega_\varepsilon')$ and satisfies in the weak sense
\begin{equation}
\label{eq:equationv}
-\nabla \cdot( \varphi^2 (\nabla v + \hat{W} v)) = 0\ \text{in}\ \Omega_{\varepsilon}',
\end{equation}
with
\begin{equation}
\label{eq:defW}
\hat{W} = W_1 - W_2.
\end{equation}
\end{lemma}
Note that the computations take care of what happens through the nodal set of $u$, i.e. $Z$.

\begin{proof}
We need to prove that for every $h \in C_0^{\infty}(\Omega_{\varepsilon}')$,
\begin{equation}
\label{eq:weakformulationv}
\int_{\Omega_{\varepsilon}'} \varphi^2 \nabla v \cdot \nabla h + \int_{\Omega_{\varepsilon}'} \varphi^2 v \hat{W} \cdot \nabla h = 0.
\end{equation}
Moreover, by \cite[Proposition 9.4]{Bre11}, we have in the weak sense that
\begin{equation}
\nabla \varphi = \nabla \varphi 1_{\Omega_{\varepsilon}},\  \nabla v = \frac{\nabla u}{\varphi} - \frac{u \nabla \varphi}{\varphi^2} 1_{\Omega_{\varepsilon}}\ \text{in}\ \Omega_{\varepsilon}',
\end{equation}
so we need to prove
\begin{equation}
\label{eq:weakformulationvBis}
\int_{\Omega_{\varepsilon}'} \left(\varphi \nabla u - u \nabla \varphi 1_{\Omega_{\varepsilon}}\right) \cdot \nabla h + \int_{\Omega_{\varepsilon}'} \varphi u \hat{W} \cdot \nabla h = 0.
\end{equation}

From the equation satisfied by $u$, and the fact that $h \varphi$ belongs to $H^1(B_2)$, we have that
\begin{equation}
\label{eq:weakformulationu}
\int_{\Omega_{\varepsilon}'} \nabla u \cdot \nabla \left(h \varphi \right) + \int_{\Omega_{\varepsilon}'} W_1 u \cdot \nabla (h \varphi) + \int_{\Omega_{\varepsilon}'} (W_2\cdot\nabla u) h \varphi  + V u \left(h \varphi\right) = 0.
\end{equation}
Then, 
\begin{multline}
\label{eq:weakformulation}
\int_{\Omega_{\varepsilon}'} \left(\varphi \nabla u \cdot \nabla h\right) + \int_{\Omega_{\varepsilon}'} \nabla u \cdot \left( \nabla \varphi1_{\Omega_{\varepsilon}}\right) h +  \int_{\Omega_{\varepsilon}'} W_1 u \cdot (\varphi \nabla h + h \nabla \varphi 1_{\Omega_{\varepsilon}}) \\
 +  \int_{\Omega_{\varepsilon}'}(W_2\cdot\nabla u) h \varphi +  \int_{\Omega_{\varepsilon}'}V u \left(h \varphi\right) = 0.
\end{multline}

Let us remark that $u h$ belongs to $H_0^1(\Omega_{\varepsilon})$ because $uh \in H^1(\Omega_{\varepsilon}) \cap C^0(\overline{\Omega_{\varepsilon}})$ and $uh = 0$ on $\partial\Omega_{\varepsilon}$, see \cite[Theorem 9.17, (i) $\Rightarrow$ (ii)]{Bre11}. So  one can use it as a test function in \eqref{Eq-varphiOeps} to get  
\begin{equation}
\label{eq:weakformulationvarphi}
\int_{\Omega_{\varepsilon}} \nabla \varphi \cdot \nabla \left(uh\right)+ \int_{\Omega_{\varepsilon}} (W_2 \varphi) \cdot \nabla \left(uh\right) + \int_{\Omega_{\varepsilon}}  (W_1 \cdot \nabla \varphi + V \varphi )\left(uh\right)= 0.
\end{equation}
Then, we subtract the two equations, recalling that $u = 0$ on $Z$ to get
\begin{multline}
\label{eq:weakformulationuvarphi}
\int_{\Omega_{\varepsilon}'} \left(\varphi \nabla u \cdot \nabla h\right) - u \nabla \varphi 1_{\Omega_{\varepsilon}} \cdot \nabla h  + \int_{\Omega_{\varepsilon}'} \nabla u \cdot \left( \nabla \varphi 1_{\Omega_{\varepsilon}}\right) h -  h \nabla u \cdot \nabla \varphi 1_{\Omega_{\varepsilon}}\\
+  \int_{\Omega_{\varepsilon}'} W_1 u \cdot (\varphi \nabla h ) + \int_{\Omega_{\varepsilon}'} W_1 u \cdot ( h \nabla \varphi 1_{\Omega_{\varepsilon}}) - W_1 \cdot \nabla \varphi 1_{\Omega_{\varepsilon}}  \left(uh \right) \\
+ \int_{\Omega_{\varepsilon}'} (W_2 \cdot \nabla u) h \varphi - \int_{\Omega_{\varepsilon}'} (W_2 \varphi) \cdot( u (\nabla h) \varphi 1_{\Omega_{\varepsilon}} + h \nabla u 1_{\Omega_{\varepsilon}}) \\
+ \int_{\Omega_{\varepsilon}'} V u \left(h \varphi\right) - V \varphi \left(uh\right)  1_{\Omega_{\varepsilon}} = 0.
\end{multline}
By using $u 1_{\Omega_{\varepsilon}'} = u 1_{\Omega_{\varepsilon}}$ and $\nabla u = 0$ almost everywhere on $Z$, see for instance \cite[Chapter 9, Comment 4 page 314]{Bre11}, we therefore deduce from \eqref{eq:weakformulationuvarphi} the expected weak formulation \eqref{eq:weakformulationvBis}, recalling the definition \eqref{eq:defW}. This concludes the proof.
\end{proof}
From now, given a vector field $B \in \R^2$, we will denote by $(B)_1$ and $(B)_2$ the first and the second coordinates of $B$, moreover we will also use the implicit identification of $B$ to the complex number $(B)_1 + i (B)_2$.

\subsection{Quasiconformal change of variable}

We then use the theory of quasiconformal mappings, which, roughly speaking, guarantees that solutions to homogeneous elliptic divergence equations behave as harmonic functions, see e.g.  \cite{AIM09}. Here, because of the drift term $\hat{W}$, we reduce the divergence elliptic equation satisfied by $v$ to a harmonic equation with a new drift term $\tilde{W}$.
\begin{lemma}
	\label{lem:harmonich}
	There exists an homeomorphic mapping $L$ of $ \overline{B(0,2)}$ into itself such that
	\\
\indent $\bullet$ $L \in H_{\mathrm{loc}}^1(B_2)$ satisfies the following Beltrami equation
	\begin{equation}
	\label{eq:equationL}
\partial_{\overline{z}} L = \mu \partial_z L\ \text{in}\ B_2,
	\end{equation}
	with $\mu \in L^{\infty}(B_2)$, satisfying $\mu = 0$ in $B_2 \setminus \Omega_{\varepsilon}'$ and
	\begin{equation}
\label{eq:beltramicoeff}
\mu = \frac{1-\varphi^2}{1+ \varphi^2} \cdot \frac{\partial_x v + i \partial_y v}{\partial_x v - i\partial_y v}\ \text{if}\ \nabla v \neq 0,\quad \mu = 0\ \text{if}\ \nabla v = 0\qquad \text{in}\ \Omega_{\varepsilon}',
\end{equation}
	\indent $\bullet$ $L$ is a $K$-quasiconformal mapping of $B_2$ into itself, with $K$ satisfying 
	\begin{equation}
	\label{K-est-quasiconf}
		1 \leq K \leq 1 + C  \left( \varepsilon^{2/(2+\delta)} \norme{W_2}_{L^{\infty}(B_2)} + \varepsilon^2 \norme{V}_{L^{\infty}(B_2)}\right),
	\end{equation}
	\indent $\bullet$ $L(0) = 0$,\\
	\indent$\bullet$ the function 
	\begin{equation}
	\label{eq:defh}
	h = v \circ L^{-1}\ \text{in}\ L(\Omega_{\varepsilon}'),
	\end{equation}
belongs to $H_{\mathrm{loc}}^1(L(\Omega_{\varepsilon}'))$ and satisfies in the weak sense
\begin{equation}
\label{eq:equationh}
-\Delta h - \nabla \cdot ( \tilde{W} h) = 0\ \text{in}\ L(\Omega_{\varepsilon}'),
\end{equation}	
with $\tilde{W}=((\tilde{W})_1, (\tilde{W})_2) \in L_{\mathrm{loc}}^2(B_2)$ defined as follows
\begin{align}
(\tilde{W})_1 &= 2 \left(\Re ( \partial_{\overline{z}} \overline{L^{-1}} ) \Re (\overset{\diamond}{W} \circ L^{-1}) - \Im ( \partial_{\overline{z}} \overline{L^{-1}} ) \Im (\overset{\diamond}{W} \circ L^{-1})\right),\label{eq:deftildeW1}\\
(\tilde{W})_2 &= 2 \left(\Re ( \partial_{\overline{z}} \overline{L^{-1}} ) \Im (\overset{\diamond}{W} \circ L^{-1}) + \Im ( \partial_{\overline{z}} \overline{L^{-1}} ) \Re (\overset{\diamond}{W} \circ L^{-1})\right). \label{eq:deftildeW2}
\end{align}
where
\begin{equation}
\label{eq:hatW}
\overset{\diamond}{W} =  \frac{ \varphi^2 (\hat{W})_1(1+\mu)}{2} + \frac{ i \varphi^2 (\hat{W})_2(1-\mu)}{2}.
\end{equation}
\end{lemma}

\begin{proof}
Let us first consider a ball $B$ of $\R^2$ contained in $\Omega_{\varepsilon}'$.

By Poincaré lemma, see for instance \cite[Section 6.5]{LMNN20}, one can then find a function $\tilde{v} \in H_{\mathrm{loc}}^1(B)$ such that $\varphi^2( \partial_x v + (\hat{W})_{1} v) = \partial_y \tilde{v}\ \text{and}\  \varphi^2( \partial_y v + (\hat{W})_{2} v) = -\partial_x \tilde{v}$. Setting $w = v + i \tilde{v}$, we easily check that $w$ is a solution to the Beltrami equation
\begin{equation}
\label{eq:beltramiequation-in-B}
	\partial_{\overline{z}} w = \mu \partial_{z} w - \overset{\diamond}{W} v \text{ in } B,
\end{equation}
with the Beltrami coefficient $\mu$ defined in \eqref{eq:beltramicoeff} and with $\overset{\diamond}{W}$ defined in \eqref{eq:hatW}.

Note that, since $\Omega_\varepsilon'$ is not simply connected, $w$ and $\tilde v$ cannot be a priori defined in the whole set $\Omega_\varepsilon'$. However, since $v$ is well-defined in $\Omega_\varepsilon'$, we can safely define the Beltrami coefficient $\mu$ by \eqref{eq:beltramicoeff} in $\Omega_{\varepsilon}'$, and we further have, by \eqref{Est-Size-tilde-varphiOeps}, 
%
\begin{equation}
\label{eq:estimatemu}
	\norme{\mu}_{L^\infty(\Omega_\varepsilon')} \leq 
	\norme{\frac{1-\varphi^2}{1+ \varphi^2}}_{L^\infty(\Omega_\varepsilon')} \leq C  \left( \varepsilon^{2/(2+\delta)} \norme{W_2}_{L^{\infty}(B_2)} + \varepsilon^2 \norme{V}_{L^{\infty}(B_2)}\right).
\end{equation}
We then extend $\mu$ by zero outside $\Omega_{\varepsilon}'$ to the whole complex plane, and remark that $\mu$ has compact support. 

We then use \cite[Theorem 5.3.2]{AIM09} to obtain the existence of a $K$-quasiconformal homeomorphism $\Psi$ of the complex plane such that $\Psi \in H_{\mathrm{loc}}^1(\C)$, $\Psi$ satisfies the Beltrami equation 
\begin{equation}
\label{eq:beltramiequation-in-C}
	\partial_{\overline{z}} \Psi = \mu \partial_{\overline{z}} \Psi \text{ in } \C,
\end{equation}
and $K = \frac{1 + \sup |\mu|}{1-\sup|\mu|}$. In our case, according to \eqref{eq:estimatemu}, we have \eqref{K-est-quasiconf}.

Since $\Psi(B_2)$ is a simply connected domain which does not fill the whole plane, by Riemann mapping theorem, see \cite[Theorem 8.2]{Bel15}, there exists $\alpha : \overline{\Psi(B_2)} \to \overline{B_2}$, one-to-one, such that $\alpha$ is holomorphic in $\Psi(B_2)$ and $\alpha(\Psi(0)) = 0$. The mapping $L =\alpha \circ \Psi $ from $B_2$ onto itself is a $K$-quasiconformal mapping from $B_2$ onto itself with $L(0)=0$. Actually, we have
\begin{equation}
\partial_{\overline{z}} L = (\partial_z \alpha \circ \Psi) \partial_{\overline{z}} \Psi + (\partial_{\overline{z}} \alpha \circ \Psi) \partial_{\overline{z}} \overline{\Psi}  \underset{\partial_{\overline{z}} \alpha = 0}{=} (\partial_z \alpha \circ \Psi) \mu \partial_{z} \Psi,
\end{equation}
and
\begin{equation}
\partial_z L = (\partial_z \alpha \circ \Psi) \partial_{z} \Psi + (\partial_{\overline{z}} \alpha \circ \Psi) \partial_{z} \overline{\Psi} = (\partial_z \alpha \circ \Psi) \partial_{z} \Psi.
\end{equation}
So, \eqref{eq:equationL} and \eqref{K-est-quasiconf} hold true.

Now, let us check the equation satisfied by $h = v \circ L^{-1}$ in $L(\Omega_{\varepsilon}')$. Let us consider a ball $B$ of $\R^2$ contained in $\Omega_{\varepsilon}'$ and let us define $\tilde{v} \in H_{\mathrm{loc}}^1(B)$ as before. Here, we use the complex chain rule and the fact that $L^{-1}$ satisfies the following Beltrami equation
\begin{equation}
- \partial_{\overline{z}} L^{-1} = (\mu \circ L^{-1}) \overline{\partial_z L^{-1}}\ \text{in}\ L(B_2) = B_2,
\end{equation}
to obtain from \eqref{eq:beltramiequation-in-B}
\begin{align*}
\partial_{\overline{z}} (w \circ L^{-1}) &= (\partial_z w \circ L^{-1}) \partial_{\overline{z}} L^{-1} + (\partial_{\overline{z}} w \circ L^{-1}) \partial_{\overline{z}} \overline{L^{-1}} \\
&= (\partial_z w \circ L^{-1}) (- \mu \circ L^{-1} \overline{\partial_{z} L^{-1}})+ (\partial_{\overline{z}} w \circ L^{-1}) \partial_{\overline{z}} \overline{L^{-1}}\\
&= (\partial_z w \circ L^{-1}) (- \mu \circ L^{-1} \overline{\partial_{z} L^{-1}}) + [\mu \circ L^{-1} (\partial_z w \circ L^{-1}) - (\overset{\diamond}{W} \circ L^{-1} )( v \circ L^{-1})]\partial_{\overline{z}} \overline{L^{-1}}\\
\partial_{\overline{z}} (w \circ L^{-1})&= - (\partial_{\overline{z}} \overline{L^{-1}}) (\overset{\diamond}{W} \circ L^{-1} )( v \circ L^{-1})\  \text{in}\ L(B).
\end{align*}
So by taking the real part and the imaginary part in both sides of the equality, we obtain in $L(B)$,
\begin{equation}
\partial_{x} (v \circ L^{-1}) - \partial_{y} (\tilde{v} \circ L^{-1}) = 2 [-\Re ( \partial_{\overline{z}} \overline{L^{-1}} ) \Re (\overset{\diamond}{W} \circ L^{-1}) +\Im ( \partial_{\overline{z}} \overline{L^{-1}} ) \Im (\overset{\diamond}{W} \circ L^{-1})] v \circ L^{-1},
\end{equation}
\begin{equation}
\partial_{y} (v \circ L^{-1}) + \partial_{x} (\tilde{v} \circ L^{-1}) = 2 [-\Re ( \partial_{\overline{z}} \overline{L^{-1}} ) \Im (\overset{\diamond}{W} \circ L^{-1}) - \Im ( \partial_{\overline{z}} \overline{L^{-1}} ) \Re (\overset{\diamond}{W} \circ L^{-1})] v \circ L^{-1}.
\end{equation}
So $v \circ L^{-1}$ is a weak solution to
\begin{equation}
\label{eq:equationvL-1}
- \Delta (v \circ L^{-1}) - \nabla \cdot (\tilde{W} v \circ L^{-1}) = 0\ \text{in}\ L(B),
\end{equation}
with $\tilde{W}$ defined in \eqref{eq:deftildeW1}, \eqref{eq:deftildeW2}. To sum up, for every ball $B \subset \Omega_{\varepsilon}'$, the equation \eqref{eq:equationvL-1} is satisfied weakly in $L(B)$, so the equation \eqref{eq:equationh} is satisfied weakly in $L(\Omega_{\varepsilon}')$ because $L$ is an homeomorphism.
\end{proof}


We conclude this part by the {\color{black}analysis of the} distortion of distances through {\color{black}the quasiconformal mapping $L$}, which is precisely given by Mori's theorem, see \cite[Chapter III, Section C]{Ahl66}: for a $K$-\textcolor{black}{quasiconformal} mapping $L$ of $B(0,R)$ into itself, for all $z_1, \, z_2 \in B(0,R)$, 
\begin{equation}
\label{eq:Ldistance}
	\frac{1}{16} \left|\frac{z_1-z_2}{R}\right|^K \leq \frac{|L(z_1)-L(z_2)|}{R} \leq 16 \left|\frac{z_1-z_2}{R}\right|^{1/K}.
\end{equation}
Here, $R=2$.

Based on this result, it is not difficult to prove that the balls of $F_\varepsilon$ are not too much distorted by the map $L$, see the lemma afterwards. 

\begin{lemma}
\label{lem:imagequasiconformal}
	Let $L$ be the mapping as in \Cref{lem:harmonich}.

	There exist a positive constant $c>0$  {\color{black}(independent of $u$, $W_1$, $W_2$ and $V$)} such that for every $1 \leq \kappa < +\infty$, $c_0>0$, $\varepsilon >0$ satisfying 
	\begin{equation}
		\label{eq:varepsquasiconform}
		  \left( \varepsilon^{2/(2+\delta)} \norme{W_2}_{L^{\infty}(B_2)} + \varepsilon^2 \norme{V}_{L^{\infty}(B_2)}\right)
		\log\left( \frac{2}{\varepsilon} \right) \leq c,
	\end{equation}	
	\indent $\bullet$ $L$ satisfies the following Lipschitz property
	\begin{equation}
\label{eq:LipschitzL}
\frac{1}{32} |z_1-z_2| \leq |L(z_1)-L(z_2)| \leq 32  |z_1-z_2|\qquad \forall \varepsilon \leq |z_1-z_2|,\ z_1, z_2 \in B_2,
\end{equation}
	\indent $\bullet$ the images of the disks $B(x_j,\varepsilon)$ (recall the definition in Lemma \ref{lem:constructionperforated}) are contained in disks of the form $B(L(x_j), 32  \varepsilon)$, indexed by $j \in J$, that are $(C_0/32 - 64)\varepsilon$-separated from each other, from $L(Z)$, from $L(x_{\max})$ and from $L(0) = 0$,\\
	\indent $\bullet$ $L^{-1} \in W^{1,\kappa}(B_{2-c_0})$ and
	\begin{equation}
	\label{eq:gradientL}
	\norme{L^{-1}}_{W^{1,\kappa}(B_{2-c_0})} \leq C_{\kappa},
	\end{equation}
\indent $\bullet$	$L(B(0,r/2))$ contains $B(0,2r')$ with 
\begin{equation}
\label{eq:defr'}
r' = c r^2\ \text{if}\ r/2 \leq \varepsilon,\qquad r' = cr\ \text{if}\ r/2 > \varepsilon.
\end{equation}
	%
	%
\end{lemma}

\begin{proof}
Let us first prove the second estimate of \eqref{eq:LipschitzL}. Indeed, by the second inequality of \eqref{eq:Ldistance}, we have
\begin{equation}
|L(z_1)-L(z_2)| \leq 16 \cdot 2^{1-1/K} |z_1-z_2|^{1/K} \leq 16 \cdot 2^{1-1/K} \varepsilon^{1/K-1} |z_1-z_2| \leq 32 |z_1-z_2|,
\end{equation}
if, using \eqref{K-est-quasiconf} and the assumption \eqref{eq:varepsquasiconform}
\begin{equation}
\label{eq:proofhelpquasiconformal}
(K-1) \log\left(\frac{2}{\varepsilon} \right) \leq C  \left( \varepsilon^{2/(2+\delta)} \norme{W_2}_{L^{\infty}(B_2)} + \varepsilon^2 \norme{V}_{L^{\infty}(B_2)}\right) \log \left(\frac{2}{\varepsilon} \right) \leq \log(2).
\end{equation}
The reverse estimate writes in the same way.

To prove that the images of the disks $B(x_j, \varepsilon)$ are contained in disks of the form $B(L(x_j),32  \varepsilon)$, we proceed as follows, for $z \in B(x_j, \varepsilon)$, we have 
\begin{equation}
|L(z)-L(x_j)| \leq 16 \cdot 2^{1-1/K} \varepsilon^{1/K} \leq 16 \cdot 2^{1-1/K} \varepsilon^{1/K-1} \varepsilon \leq 32  \varepsilon.
\end{equation}
by \eqref{eq:proofhelpquasiconformal}.

 Moreover, by using the first Lipschitz estimate, and the fact that the $x_j$ are $C_0 \varepsilon$ separated, the centers $L(x_j)$ are $C_0 \varepsilon/32$ separated. Therefore the disks $B(L(x_j),32  \varepsilon)$ are thus $(C_0/32 - 2 \cdot 32)\varepsilon$ separated. Using similar arguments, we can also prove that the disks $B(L(x_j), 32  \varepsilon)_{j \in J}$ are $(C_0/32 - 64)\varepsilon$ separated from $L(Z)$, and from $L(x_{\max})$, and from $L(0)=0$.

Moreover, $L^{-1}$ satisfies the following Beltrami equation
\begin{equation}
- \partial_{\overline{z}} L^{-1} = \mu \overline{\partial_z L^{-1}}.
\end{equation}
Therefore, $L^{-1}$ is a $K$-quasiconformal mapping and then one can use Cacciopoli's estimates from \cite[Theorem 5.4.3]{AIM09} to get that $L^{-1} \in W^{1,\kappa}(B_{2-c_0})$ because of \eqref{K-est-quasiconf} and the assumption \eqref{eq:varepsquasiconform}, with 
	\begin{equation}
\norme{D L^{-1}}_{L^{\kappa}(B_{2-c_0})} \leq C  \norme{ L^{-1}}_{L^{\infty}(B_{2})} \leq C,
	\end{equation}
where $D L^{-1}$ is the Jacobian matrix of $L^{-1}$.

If $r/2 \leq \varepsilon$, then $L(B(0,r/2))$ may have radius significantly smaller than $r$, however, by using  Mori's estimate that is \eqref{eq:Ldistance}, $L(B(0,r/2))$ contains a ball $B(0,2r^*)$ with
\begin{equation*}
r^* = 2 \left(\frac{r}{64}\right)^K \geq c r^2.
\end{equation*}
In particular, $L(B(0,r/2))$ contains a ball $B(0, 2r')$ with $r' = cr^2$.

On the other hand, by employing similar arguments as before, if $r/2 > \varepsilon$, then it is not difficult to obtain that $L(B(0,r/2))$ contains a ball $B(0,2r')$ with $r' = cr$.
\end{proof}


\textbf{Setting of parameters.} Before ending this step of the proof, we now set $\varepsilon>0$ such that
\begin{equation}
	\label{eq:epsPetitPoincareGeneralFinStep2}
	\varepsilon + \varepsilon^{2/(2+\delta)} \norme{W_1}_{L^{\infty}(B_2)} + \varepsilon^{2/(2+\delta)} \log\left(\frac{2}{ \varepsilon} \right)  \norme{W_2}_{L^{\infty}(B_2)} + \varepsilon^2 \log\left(\frac{2}{ \varepsilon} \right) \norme{V}_{L^{\infty}(B_2)} \leq c.
\end{equation} 
We then set $\varepsilon' = 32 \varepsilon$ and remark that by construction, and recalling the choice $C_0 = 18 \cdot 32^{2}$, for which we have $C_0/32 - 64 = 16\cdot 32$, the disks $B(L(x_j), \varepsilon')$ given by Lemma \ref{lem:imagequasiconformal} are  $16 \varepsilon'$-separated from each other, from $L(Z)$, from ${\partial B_2}$, from $L(0)=0$ and from $L(x_{\max})$. We will also use the notation $x_j' = L(x_j)$.

We also set 
\begin{equation}
\label{eq:defkappa}
\kappa=4/\delta +2,
\end{equation}
and 
\begin{equation}
\label{eq:defc0}
c_0 = 2^{-10}.
\end{equation}
This choice will be made clearer later.\\

From \eqref{eq:gradientL}, $\tilde{W}$ defined in \eqref{eq:deftildeW1}, \eqref{eq:deftildeW2} and \eqref{eq:hatW} belongs to $L^\kappa(B_{2-c_0})$ and we have
\begin{equation}
\label{eq:lpboundtildeW}
\norme{\tilde{W}}_{L^\kappa(B_{2-c_0})} \leq C \norme{\hat{W}}_{\infty}.
\end{equation}

\subsection{Approximate stream function} 

The goal of this part is to obtain a Poincaré lemma for the divergence free vector field $\nabla h + \tilde{W} h$, see \eqref{eq:equationh}. The main difficulty is that $L(\Omega_{\varepsilon}')$ is not simply connected, because of the perforated disks $B(x_j', \varepsilon')$. This is why we first introduce a cut-off function near these disks and near $\partial B_2$ in order to state an \textit{approximate} type Poincaré lemma.\medskip

We introduce a smooth cut-off function $\sigma$ taking value $0$ on $B(0,3)$ and value $1$ on $\R^2 \setminus B(0,4)$, and another smooth cut-off function $\xi$ such that $\xi = 1$ in $B(0,2-8 c_0)$ and $\xi=0$ in $\R^2 \setminus B(0,2-4 c_0)$ and set
\begin{equation}
\label{eq:defyinfunctionh}
	\chi (x) = \xi(x) \prod_{ j \in J}  \sigma \left( \frac{x - x_j'}{\varepsilon'} \right) \text{ for } x \in \R^2.
\end{equation}
We have the following lemma, coming directly from \eqref{eq:equationh} and the second point of \Cref{lem:imagequasiconformal}.
\begin{lemma}
	\label{lem:divergencepcutoff}
	The function $\chi h$ defined in $\R^2$ belongs to $H^1(\R^2)$ and satisfies in the weak sense
\begin{equation}
\label{eq:equationchiv}
-\nabla \cdot(\chi  (\nabla h + \tilde{W} h)) = - \nabla \chi \cdot (\nabla  h + \tilde{W} h) \ \text{in}\ \R^2.
\end{equation}
\end{lemma}

For the next, we need to recast the divergence elliptic equation in polar coordinates. For that purpose, we recall well-known useful relations between Cartesian coordinates and polar coordinates in the following paragraph.
\medskip

\textbf{From Cartesian to polar coordinates.} By taking $e_1 = (1,0)$ and $e_2=(0,1)$ the canonical basis of $\R^2$, for $(\rho, \theta) \in [0,+\infty) \times [0, 2 \pi)$,  we set $e_{\rho} = \cos(\theta) e_1 + \sin(\theta) e_2,\ e_{\theta} = - \sin(\theta) e_1 + \cos(\theta) e_2,$ therefore we have the relation $e_1 = \cos(\theta) e_{\rho} - \sin(\theta) e_{\theta},\ e_2 = \sin(\theta) e_{\rho} + \cos(\theta) e_{\theta}$. We then have that $(e_{\rho}, e_{\theta})$ is an orthonormal basis of $\R^2$.

Given a function $u \in C^1(\R^2)$, we then define
\begin{equation}
	U(\rho,\theta) = u(\rho \cos (\theta), \rho \sin (\theta)),\qquad (\rho, \theta) \in[0,+\infty) \times [0, 2 \pi).
\end{equation}
By the chain rule, we have the following relation for $\nabla u = \partial_{x_1} u e_1 + \partial_{x_2} u e_2$, 
\begin{equation}
\nabla u = \partial_{\rho} U e_{\rho} + \frac{1}{\rho} \partial_{\theta} U e_{\theta},
\end{equation}
and for $\text{curl}(u) = \partial_{x_2} u e_1 - \partial_{x_1} u e_2$,
\begin{equation}
\text{curl}( u )= \frac{1}{\rho} \partial_{\theta} U e_{\rho} -  \partial_{\rho} U e_{\theta}.
\end{equation}
Given now a vector valued-function $g \in C^1(\R^2;\R^2)$, setting 
$$g(\rho \cos (\theta), \rho \sin (\theta)) = G_{\rho}(\rho,\theta) e_{\rho} + G_{\theta}(\rho,\theta)  e_{\theta},$$ 
we then have by the chain rule applied to $\nabla \cdot g =\partial_{x_1} g_1 + \partial_{x_2} g_2$ that
\begin{equation}
\nabla \cdot g = \frac{1}{\rho} \partial_\rho \left(\rho G_{\rho}\right) + \frac{1}{\rho}\partial_\theta \left( G_{\theta}\right).
\end{equation}
In the following, for the sake of simplicity, we will make an abuse of notation by identifying $u$ with $U$, and $g$ with $G$.\medskip

We have the following result. 
\begin{lemma}
\label{lem:approximatestreampolar}
Let us define for $(\rho,\theta) \in [0,2)\times[0,2\pi)$,
\begin{equation}
\label{eq:defvrhotheta}
\tilde{h}(\rho,\theta) = - \int_{0}^{\rho}  \chi(s,\theta) \left[ \frac{1}{s} \partial_\theta h(s, \theta) - [\tilde{W}_1 h](s, \theta) \sin(\theta) + [\tilde{W}_2 h](s, \theta) \cos(\theta)\right] ds.
\end{equation}
Then, $\tilde{h} \in H^1(B_2)$ and satisfies for $(\rho,\theta) \in [0,2)\times[0,2\pi)$,
\begin{align}
\label{eq:derivativerhov}
\partial_{\rho} \tilde{h}(\rho,\theta) &= -  \chi(\rho,\theta) \left[ \frac{1}{\rho} \partial_\theta h(\rho, \theta) - [\tilde{W}_1 h](\rho, \theta) \sin(\theta) + [\tilde{W}_2 h](\rho, \theta) \cos(\theta)\right],\\
\partial_{\theta} \tilde{h}(\rho,\theta) &= \rho \chi(\rho,\theta)  \left[  \partial_\rho h(\rho, \theta) + [\tilde{W}_1 h](\rho, \theta) \cos(\theta) + [\tilde{W}_2 h](\rho, \theta) \sin(\theta)\right] + E_h, \label{eq:derivativethetav}
\end{align}
where
\begin{multline}
E_h(\rho,\theta) = - \int_{0}^{\rho}\partial_{\theta} (\chi) (s,\theta) \left[ \frac{1}{s} \partial_\theta h(s, \theta) - [\tilde{W}_1 h](s, \theta) \sin(\theta) + [\tilde{W}_2 h](s, \theta) \cos(\theta)\right] ds\\
 - \int_{0}^{\rho} s \partial_{\rho} (\chi)(s,\theta)  \left[ \partial_\rho h(s,\theta) + [\tilde{W}_1 h](s,\theta) \cos(\theta) + [\tilde{W}_2 h](s,\theta) \sin(\theta)\right]  ds.
 \label{eq:defEh}
\end{multline}
\end{lemma}

The function $\tilde{h}$ is called an approximate stream function of $\chi (\nabla h + \tilde{W} h)$ because \eqref{eq:derivativerhov} and \eqref{eq:derivativethetav} translates into
\begin{equation}
\label{eq:approximatestreamcompactway}
\chi (\nabla h + \tilde{W} h) = \mathrm{curl}(\tilde{h}) - \frac{ E_h}{\rho} e_{\rho}.
\end{equation}
Formally, \Cref{lem:approximatestreampolar} could be justified as follows. Taking \eqref{eq:defvrhotheta}, we immediately obtain \eqref{eq:derivativerhov}. Moreover, by writing \eqref{eq:equationchiv} in polar coordinates, we have
\begin{multline}
\label{eq:divergenceeqpolar}
- \partial_{\rho} \left(\chi(\rho,\theta) \rho \left[  \partial_\rho h(\rho, \theta) + [\tilde{W}_1 h](\rho, \theta) \cos(\theta) + [\tilde{W}_2 h](\rho, \theta) \sin(\theta)\right]\right)\\
- \partial_{\theta} \left(\chi(\rho,\theta) \left[ \frac{1}{\rho} \partial_\theta h(\rho, \theta) - [\tilde{W}_1 h](\rho, \theta) \sin(\theta) + [\tilde{W}_2 h](\rho, \theta) \cos(\theta)\right] \right)\\
= - \rho  \partial_{\rho} \chi(\rho,\theta) \left(  \partial_\rho h(\rho, \theta) + [\tilde{W}_1 h](\rho, \theta) \cos(\theta) + [\tilde{W}_2 h](\rho, \theta) \sin(\theta)\right)\\
-    \partial_{\theta} \chi(\rho,\theta) \left(  \frac{1}{\rho} \partial_\theta h(\rho, \theta) - [\tilde{W}_1 h](\rho, \theta) \sin(\theta) + [\tilde{W}_2 h](\rho, \theta) \cos(\theta)\right)\ \text{in}\ B_2.
\end{multline}
If \eqref{eq:divergenceeqpolar} was satisfied in the strong sense, we obtain from \eqref{eq:defvrhotheta} and an integration by parts the desired equality \eqref{eq:derivativethetav}. The following proof takes care of the fact that \eqref{eq:divergenceeqpolar} is satisfied a priori only in the weak sense by using a regularization argument.
\begin{proof}
The equation \eqref{eq:equationchiv} rewrites 
\begin{equation}
\label{eq:divFequalf}
- \nabla \cdot F = f \ \text{in}\ \R^2,
\end{equation}
with
\begin{equation}
F = \chi  (\nabla h + \tilde{W} h),\ f = -  \nabla \chi \cdot (\nabla h + \tilde{W} h).
\end{equation}
We consider $(K_{n})_{n \geq 1}$ a radial approximation of the identity, or a standard mollifier, satisfying
\begin{equation}
K_n \in C^{\infty}(\R^2),\ K_n \geq 0,\ \int_{\R^2} K_n dx = 1,\ K_n(x) = K_n(|x|),\ \text{supp}(K_n) \in B(0,1/n)\quad \forall n \geq 1.
\end{equation}
Then $F_{n} =  K_n * F \in C^{\infty}(\R^2)$, satisfies in the weak sense (then in the strong sense)
\begin{equation}
\label{eq:divergencesmooth}
- \nabla \cdot F_{n} =  K_n * f =: f_n\ \text{in} \ \R^2.
\end{equation}
Indeed, by taking a test function $\zeta \in C_c^{\infty}(\R^2)$, we have by Fubini's theorem and \eqref{eq:divFequalf}
\begin{multline}
\int_{\R^2} F_n \cdot \nabla \zeta = \int_{\R^2} (K_n * F) \cdot \nabla \zeta = \int_{\R^2}  F \cdot K_n * \nabla  \zeta \\
=  \int_{\R^2}  F \cdot \nabla(K_n *  \zeta) = \int_{\R^2} f (K_n *  \zeta) =  \int_{B_2} (K_n * f) \zeta = \int_{\R^2} f_n \zeta.
\end{multline}
First, in polar coordinates, \eqref{eq:divergencesmooth} becomes
\begin{equation}
\label{eq:divergencesmoothpolar}
- \partial_{\rho} (\rho F_{n,\rho})\\
- \partial_{\theta} F_{n,\theta}\\
= \rho f_n \ \text{in} \ \R^2.
\end{equation}
Therefore, let us define
\begin{equation}
\tilde{h}_{n}(\rho,\theta) = - \int_{0}^{\rho} F_{n,\theta}(s, \theta) ds\ \text{in}\ \R^2.
\end{equation}
We then have
\begin{equation}
\partial_{\rho} \tilde{h}_{n}(\rho,\theta) = - F_{n,\theta}(\rho, \theta)\ \text{in}\ \R^2.
\end{equation}
And by using \eqref{eq:divergencesmoothpolar}, we have
\begin{multline}
\partial_{\theta} \tilde{h}_{n}(\rho,\theta) =  - \int_{0}^{\rho} \partial_{\theta} F_{n,\theta}(s, \theta) ds =  \int_{0}^{\rho} \partial_{s} (s F_{n,\theta}(s, \theta)) ds + \int_{0}^{\rho}  s f_n(s, \theta) ds \\
= \rho F_{n,\theta}(\rho,\theta) + \int_{0}^{\rho}  s f_n(s,\theta) ds\ \text{in}\ \R^2.
\end{multline}
On the other hand, by sending $n \to +\infty$, we get by convolution properties, see for instance \cite[Theorem 4.22]{Bre11} that
\begin{align*}
\lim_{n \to +\infty} \norme{F_n-F}_{L^2(\R^2)} &= 0,\\
\lim_{n \to +\infty} \norme{f_n-f}_{L^2(\R^2)} &= 0.
\end{align*}
Moreover, let us set
\begin{equation}
\tilde{h}(\rho,\theta) = - \int_{0}^{\rho} F_{\theta}(s, \theta) ds\ \text{in}\ \R^2.
\end{equation}
First, we have
\begin{equation}
\lim_{n \to +\infty} \norme{\tilde{h}_n-\tilde{h}}_{L_{\mathrm{loc}}^2(\R^2)} = 0,
\end{equation}
Moreover, we have that
\begin{align*}
\lim_{n \to +\infty} \norme{\partial_{\rho} \tilde{h}_n+ F_{\theta}(\rho,\theta)}_{L_{\mathrm{loc}}^2(\R^2)} &= 0,\\
\lim_{n \to +\infty} \norme{\frac{1}{\rho} \partial_{\theta} \tilde{h}_n-\left( F_{\rho}(\rho,\theta) + \frac{1}{\rho}\int_{0}^{\rho}  s f(s,\theta) ds \right)}_{L_{\mathrm{loc}}^2(\R^2)} &= 0,
\end{align*}
We then deduce from the last three limits and the uniqueness of the limit in the distributional sense that $\tilde{h} \in H_{\mathrm{loc}}^1(\R^2)$ and
\begin{align}
\partial_{\rho} \tilde{h}(\rho,\theta) &= -F_{\theta}(\rho,\theta)\ \text{in}\ \R^2,\\
\partial_{\theta} \tilde{h}(\rho,\theta) 
&= \rho F_{\theta}(\rho,\theta) + \int_{0}^{\rho}  s f(s,\theta) ds\ \text{in}\ \R^2.
\end{align}
By observing that for $(\rho, \theta) \in[0,+\infty) \times [0, 2 \pi)$ we have
\begin{equation}
F_{\theta}(\rho,\theta) = \chi(\rho,\theta) \left[ \frac{1}{\rho} \partial_\theta h(\rho, \theta) - [\tilde{W}_1 h](\rho, \theta) \sin(\theta) + [\tilde{W}_2 h](\rho, \theta) \cos(\theta)\right],
\end{equation}
\begin{equation}
F_{\rho}(\rho,\theta) = \chi(\rho,\theta) \rho \left[  \partial_\rho h(\rho, \theta) + [\tilde{W}_1 h](\rho, \theta) \cos(\theta) + [\tilde{W}_2 h](\rho, \theta) \sin(\theta)\right],
\end{equation}
and
\begin{multline}
\rho f(\rho,\theta) = -\rho  \partial_{\rho}\chi(\rho,\theta)  \left(  \partial_\rho h(\rho, \theta) + [\tilde{W}_1 h](\rho, \theta) \cos(\theta) + [\tilde{W}_2 h](\rho, \theta) \sin(\theta)\right)\\
-   \partial_{\theta} \chi(\rho,\theta) \left(  \frac{1}{\rho} \partial_\theta h(\rho, \theta) - [\tilde{W}_1 h](\rho, \theta) \sin(\theta) + [\tilde{W}_2 h](\rho, \theta) \cos(\theta)\right),
\end{multline}
we then obtain the conclusion of the proof by specifying $(\rho, \theta) \in[0,2) \times [0, 2 \pi)$.
\end{proof}
\subsection{The reduction to the non-homogeneous $\partial_{\overline{z}}$-equation}

The goal of this last part in Step 3 is to use the approximate stream function $\tilde{h}$ defined in the previous subsection in order to simplify a bit more the equation \eqref{eq:equationh}. In order to do that, we will first computed the $\partial_{\overline{z}}$-equation satisfied by $\chi h + i \tilde{h}$, then reduce it by a Cauchy transform.\medskip

We have the following Beltrami equation, perturbed by a zero order term.
\begin{lemma}
\label{lem:beltramifirstequation}
Let us define
\begin{equation}
\label{eq:defgamma}
\gamma = \chi h + i \tilde{h}\ \text{in}\ B_2.
\end{equation}
Then $\gamma \in H^1(B_2)$ and satisfies the following Beltrami equation
\begin{equation}
\label{eq:equationgamma}
\partial_{\overline{z}} \gamma  = \alpha \gamma +(\partial_{\overline{z}}\chi)   h  + \tilde{E}_h\ \text{in}\ B_2,
\end{equation}
where for $(\rho,\theta) \in [0,2)\times[0,2\pi)$,
\begin{align}
\alpha &= \Big[\frac{\chi e^{i \theta} (-\tilde{W}_1 \cos(\theta) - \tilde{W}_2 \sin(\theta) + i \tilde{W}_1 \sin(\theta) - i \tilde{W}_2 \cos(\theta)) }{4}\Big] \left(1 + \frac{\overline{\gamma}}{\gamma}\right)\ \text{if}\ \gamma \neq 0,\label{eq:defwtilde}\\
 \alpha &= 0\ \text{if}\ \gamma = 0,\label{eq:defwtilde0}\\
\tilde{E}_h &= - \frac{e^{i \theta}E_h}{2\rho}.\label{eq:ehtilde}
\end{align}
Moreover, we have
\begin{equation}
\label{eq:boundalpha}
\norme{\alpha}_{L^\kappa(B_2)} \leq C \norme{W_1}_{\infty} + C \norme{W_2}_{\infty}.
\end{equation}
\end{lemma}
\begin{proof}
We use the fundamental relation
\begin{equation}
\partial_{\overline{z}} f = \frac{1}{2} e^{i \theta} \left( \partial_{\rho}f  + \frac{i}{\rho} \partial_{\theta} f\right).
\end{equation}

We compute
\begin{equation}
2 \partial_{\overline{z}} \gamma = 2 (\partial_{\overline{z}} \chi) h + \chi e^{i \theta} \left( \partial_{\rho} h  + \frac{i}{\rho} \partial_{\theta} h\right) + i e^{i \theta} \left( \partial_{\rho} \tilde{h}  + \frac{i}{\rho} \partial_{\theta} \tilde{h}\right).
\end{equation}

Now, we use the equation satisfied by $\tilde{h}$, i.e. \eqref{eq:derivativerhov} and \eqref{eq:derivativethetav} and we get
\begin{align}
\partial_{\rho} \tilde{h} &= - \chi \left[ \frac{1}{\rho} \partial_\theta h - \tilde{W}_1 h \sin(\theta) + \tilde{W}_2 h \cos(\theta)\right],\\
 \partial_{\theta} \tilde{h} &= \chi \rho  \left[ \partial_\rho h + \tilde{W}_1 h \cos(\theta) + \tilde{W}_2 h \sin(\theta)\right] + E_h,
\end{align}
then
\begin{multline}
2 \partial_{\overline{z}} \gamma - 2 (\partial_{\overline{z}} \chi) h\\
 = 
- \chi e^{i \theta} \tilde{W}_1 h \cos(\theta) - \chi  e^{i \theta} \tilde{W}_2 h \sin(\theta) + i \chi  e^{i \theta} \tilde{W}_1 h \sin(\theta) - i \chi  e^{i \theta} \tilde{W}_2 h \cos(\theta) - e^{i \theta} \frac{E_h}{\rho}.
\end{multline}
Hence, we obtain easily the expected equation \eqref{eq:equationgamma}.

The final bound \eqref{eq:boundalpha} comes from \eqref{eq:defwtilde}, \eqref{eq:defwtilde0}, \eqref{eq:lpboundtildeW} and the properties of the cut-off $\chi$ defined in \eqref{eq:defyinfunctionh}.
\end{proof}

Let us define the operator
\begin{equation}
\label{eq:defT} 
T \omega(z) = - \frac{1}{\pi} \int_{B_2} \frac{\omega(\zeta)}{\zeta-z} d \zeta,\qquad \forall \omega \in L^P(B_2),\ P>2,
\end{equation}
that is called the Cauchy transform of $\omega$.

We have the following result.
\begin{lemma}
\label{lem:bojarski}
Let $\omega \in L^P(B_2)$, $P>2$. Then, $T \omega$ exists almost everywhere as an absolutely convergent integral Moreover, the following relations and estimates hold
\begin{align}
\partial_{\overline{z}} (T \omega) &= \omega,\label{eq:dzbarT}\\
|T \omega (z)| &\leq C(P) \norme{\omega}_{L^P(B_2)},\label{eq:boundT} \\
|T \omega (z_1) - T \omega (z_2)| & \leq C(P)\norme{\omega}_{L^P(B_2)} |z_1-z_2|^{Q},\ Q=(P-2)/P,\ P < +\infty,\label{eq:HolderT}
\end{align}
\end{lemma}
All these results of \Cref{lem:bojarski} are collected in \cite[Section 1]{Boj57}, \cite[Chapter 1, Paragraph 6]{Vek62}.\medskip 

\textbf{Setting of parameters.} We now fix
\begin{equation}
\label{eq:defP}
P = \kappa = 4/\delta + 2,
\end{equation}
and
\begin{equation}
\label{eq:defQ}
Q =  (P-2)/P = 2/(2+\delta).
\end{equation}

An application of the previous lemma gives the following result.

\begin{lemma}
\label{lem:defomegaandT}
There exists $\beta \in L^{\infty}(B_2)$ such that
\begin{equation}
\label{eq:equationomega}
\partial_{\overline{z}}  \beta  = \alpha.
\end{equation}
Moreover, $\beta \in L^{\infty}(B_2)$ satisfies
\begin{equation}
\label{eq:boundbeta}
\norme{\beta}_{L^{\infty}(B_2)} \leq C \norme{\hat{W}}_{\infty}.
\end{equation}
Finally, we have the following Hölder's estimate
\begin{equation}
\label{eq:HolderTSpecific}
|\beta (z_1) - \beta (z_2)|  \leq C \norme{\hat{W}}_{\infty} |z_1-z_2|^{Q}.
\end{equation}
\end{lemma}
\begin{proof}
Set $\beta = T \alpha$, and use \eqref{eq:dzbarT}, \eqref{eq:boundT} and \eqref{eq:HolderT}.
\end{proof}

Finally, we have the following non-homogeneous $\partial_{\overline{z}}$-equation.
\begin{lemma}
\label{lem:beltramiequationzeta}
Let us define 
\begin{equation}
\label{eq:defzeta}
\zeta = \exp(- \beta) \gamma \ \text{in}\ B_2.
\end{equation}
Then, we have $\zeta \in H^1(B_2)$ and satisfies the following Beltrami equation
\begin{equation}
\label{eq:equationzeta}
\partial_{\overline{z}} \zeta = \exp(- \beta) [(\partial_{\overline{z}} \chi) h  + \tilde{E}_h]  =: F \ \text{in}\ B_2.
\end{equation}
\end{lemma}
\begin{proof}
This directly comes from \eqref{eq:equationomega} and \eqref{eq:equationgamma}.
\end{proof}

\section{Step 3: The Carleman estimate to the non-homogeneous $\partial_{\overline{z}}$-equation}
\label{sec:carlemanpart}

The goal of this section is to apply a suitable $L^2$-Carleman estimate to the equation satisfied by $\zeta$, see \Cref{lem:beltramiequationzeta} above, in order to deduce the vanishing order estimate for $u$, that is \eqref{eq:observabilityu}. The source terms in \eqref{eq:equationzeta} will be absorbed by the left hand side term of the Carleman estimate, by taking the $s$-parameter sufficiently big in function of $\varepsilon$. The boundary terms will be absorbed by using the assumption on $u$, i.e. \eqref{eq:ucontrolledneartheboundary}, and by taking the $s$-parameter sufficiently big in function of $K$. In order to deduce from the $L^2$-Carleman estimate a $L^{\infty}$-bound on $u$, that is an estimate of $|u(x_{\max})|$, we will finally use maximal regularity estimates for the operator $\partial_{\overline{z}}$.

\subsection{Preliminaries in prevision of the Carleman estimate}

The goal of this first part is to state an elementary $L^2$-Carleman estimate in the two-dimensional setting and to prove very useful estimates for the absorption of the source term in \eqref{eq:equationzeta}. \medskip

For $s \geq 1$, a parameter, let us introduce the notation
\begin{equation}
\psi_s(z)  = - s \log(|z|) + |z|^2.
\end{equation}

First, remark that for every $z \neq 0$,
\begin{equation}
\Delta \psi_s (z) \geq 2.
\end{equation}
We have the following Carleman estimate, \cite[Proposition 2.2]{DF90}.
\begin{prop}
\label{prop:carlemanL2}
Then for every $y \in C_c^{\infty}(B_2\setminus\{0\})$, we have
\begin{equation}
\label{eq:Carlemany}
 \int_{B_2} |y|^2  e^{2 \psi_s(z)} dz \leq C \int_{B_2}|\bar{\partial} y|^2 e^{2 \psi_s(z)} dz.
\end{equation}
\end{prop} 

For the next, let us introduce the following notation. Let $a \in B_2$ and $0<r_1<r_2$, then we denote by $C(a,r_1,r_2) = \{ z \in B_2\ ;\ r_1 < |z-a| < r_2\}$, i.e. the annulus centered in $a$ with inner radius $r_1$ and outer radius $r_2$.

In order to deal with the local term in the next application of the Carleman estimate, we need the following lemma.
\begin{lemma}
\label{lem:estimationlocCarleman}
For $z' \in B(x_j', 4 \varepsilon')$, $z \in C(x_j', 6 \varepsilon', 8 \varepsilon')$ with $|z| \leq |z'| - 2 \varepsilon'$, the following estimates hold
\begin{align}
&\frac{1}{C} |h (z)| \leq |h(z')| \leq C  |h(z)|,\label{eq:harnacklocal}\\
&\frac{1}{C} |\exp(-\beta (z))| \leq |\exp(-\beta (z'))| \leq C  |\exp(-\beta (z))|,\label{eq:expTwLlocal}\\
& \exp(2\psi_s(z')) \leq C \exp(-cs \varepsilon) \exp(2\psi_s(z)),\label{eq:carlemanweightlocal}
\end{align}
\end{lemma}
Heuristically, \Cref{lem:estimationlocCarleman} tells us that $h$ and $\exp(-\beta)$ do not vary too much near the disks $B(x_j',\varepsilon')$, this is the purpose of \eqref{eq:harnacklocal} and \eqref{eq:expTwLlocal} while there is ball contained in $C(x_j', 6 \varepsilon', 8 \varepsilon')$ where the Carleman weight is $\exp(c s \varepsilon)$-bigger than the Carleman weight in $B(x_j',4 \varepsilon')$, this is the purpose of \eqref{eq:carlemanweightlocal}.
\begin{proof}
We use the fact that the disks $B(x_j', \varepsilon')$ given by Lemma \ref{lem:imagequasiconformal} are  $16 \varepsilon'$-separated from $L(Z)$. So $L^{-1}(B(x_j', 10  \varepsilon')) \cap Z = \emptyset$. Moreover, from \eqref{eq:LipschitzL} and $|z-z'| \geq 2 \varepsilon' = 64 \varepsilon$, we have
$$ |L^{-1}(z)-L^{-1}(z')| \leq 32 |z-z'| \leq C \varepsilon,$$
so by a direct application of the Harnack's inequality of \cite[Theorem 8.20]{GT83} (or see \Cref{lem:harnackscaleeps}), by using the property
\begin{equation}
\varepsilon + \varepsilon \|W_1\|_{\infty} + \varepsilon \|W_2\|_{\infty} + \varepsilon^2 \|V\|_{\infty} \leq c,
\end{equation}
coming from \eqref{eq:epsPetitPoincareGeneralFinStep2}, we obtain
\begin{equation}
\frac{1}{C} |u\circ L^{-1} (z)| \leq |u\circ L^{-1}(z')| \leq C  |u\circ L^{-1}(z)|,
\end{equation}
so the same type of estimate for $h$, i.e. \eqref{eq:harnacklocal} by using the definitions of $h$ in \eqref{eq:defh} and $v$ in \eqref{Def-v}, together with the estimate on $\varphi$ in \eqref{Est-Size-tilde-varphiOeps}.

The second estimate \eqref{eq:expTwLlocal} comes from the Hölder's estimate on $\beta$ i.e. \eqref{eq:HolderTSpecific} and the Lipschitz estimate at scale $\varepsilon$, i.e. \eqref{eq:LipschitzL}, indeed we have from the definition of $Q$ in \eqref{eq:defQ} and \eqref{eq:epsPetitPoincareGeneralFinStep2},
\begin{equation}
|\beta(z) - \beta (z')| \leq C \|\hat{W}\|_{\infty} |z-z'|^{Q}   
\leq C \varepsilon^{Q} \|\hat{W}\|_{\infty} \leq c,
\end{equation}
so 
\begin{equation}
|\exp(\beta(z) - \beta(z'))| \leq \exp(|\beta(z) - \beta (z')|) \leq \exp(c).
\end{equation}

The third estimate \eqref{eq:carlemanweightlocal} comes from
\begin{align*}
\exp(-s \log(|z'|)) &= \exp(-s (\log(|z'|) - \log(|z|))) \exp(-s \log(|z|))\\
&\leq \exp(-s (\log(|z|+2\varepsilon') - \log(|z|)))  \exp(-s \log(|z|))\\
& \leq \exp (-s \log(1 + 2 \varepsilon'/|z|)) \exp(-s \log(|z|)) \\
&  \leq \exp(-  s \varepsilon') \exp(-s \log(|z|)).
\end{align*}

This concludes the proof.
\end{proof}
In order to deal with the non-local term in the next application of the Carleman estimate, we need the following lemma.
\begin{lemma}
\label{lem:estimationnonlocCarleman}
For $\varepsilon' \leq |z|$, $z'= z + t \frac{z}{|z|}$ with $\varepsilon' \leq t \leq 2$, the following estimate holds
\begin{align}
\label{eq:estimationbetanonlocal}
|\exp(-\beta (z'))| &\leq \exp(C\|\hat{W}\|_{\infty}t^Q) |\exp(-\beta (z))|,\\
 \exp(2\psi_s(z')) &\leq C \exp(-cs t) \exp(2\psi_s(z)).\label{eq:estimationpoidsnonlocal}
 \end{align}
Moreover we have
\begin{align}
\label{eq:estimationgradientcutoff}
 \norme{\nabla h}_{L^{p}(B(x_j', 4 \varepsilon'))} &\leq C \varepsilon^{-1} \norme{h}_{L^{\infty}(B(x_j', 8 \varepsilon'))}\qquad \forall 1 \leq p < +\infty,\\
 \norme{\nabla h}_{L^p(C(0,2-16 c_0,2-4 c_0))} & \leq C \norme{\hat{W}}_{\infty} \norme{h}_{L^{\infty}(B(0,2-2 c_0))}\qquad \forall 1 \leq p < +\infty.\label{eq:estimationgradientcutoffboundary}
\end{align}
\end{lemma}
Heuristically, the first part of \Cref{lem:estimationnonlocCarleman} tells us how much $\exp(-\beta(z'))$, respectively $\psi_s(z')$, increases, respectively decreases, along the line $z'= z + t \frac{z}{|z|}$. This exponential decreasing of the Carleman weight would be the key point to compensate the exponential increasing of the multiplier $\exp(-\beta)$, for absorbing the nonlocal term.
\begin{proof}
The first estimate comes from the Hölder's estimate on $\beta$, indeed we have from \eqref{eq:HolderTSpecific}
\begin{equation}
|\beta (z) - \beta(z')| \leq C \|\hat{W}\|_{\infty} | z- z'|^{Q}   
\leq C  \|\hat{W}\|_{\infty} t^Q,
\end{equation}
so 
\begin{equation}
|\exp(\beta (z) - \beta(z'))| \leq \exp(|\beta (z) - \beta(z')|) \leq \exp(C\|\hat{W}\|_{\infty} t^Q).
\end{equation}
The last estimate comes from
\begin{align*}
\exp(-s \log(|z'|)) &
= \exp(-s (\log(|z'|) - \log(|z|))) \exp(-s \log(|z|))\\
&\leq \exp(-s (\log(|z+t  \frac{z}{|z|}|) - \log(|z|)))  \exp(-s \log(|z|))\\
& \leq \exp (-s \log(1 + t/|z|)) \exp(-s \log(|z|)) \\
&  \leq \exp(- 2^{-1} s t) \exp(-s \log(|z|)).
\end{align*}

To prove the inequality \eqref{eq:estimationgradientcutoff}, we apply \Cref{lem:gradientdivterm} to $h$, satisfying \eqref{eq:equationh} together with \eqref{eq:lpboundtildeW} replacing $\kappa$ by $1 \leq p < +\infty$, see \Cref{lem:imagequasiconformal}. 

In the same spirit, the inequality \eqref{eq:estimationgradientcutoffboundary} is an application of \Cref{lem:gradientdivtermscale1}.

This concludes the proof.
\end{proof}

\subsection{A Carleman estimate to a non-homogeneous $\partial_{\overline{z}}$ equation}

The goal of this part is now to apply the Carleman estimate from \Cref{prop:carlemanL2} to a cut-off version of $\zeta$, and then use \Cref{lem:estimationlocCarleman}, \Cref{lem:estimationnonlocCarleman} to absorb the source term of the equation satisfied by $\zeta$ in \eqref{eq:equationzeta}.

We split $J$ such that
\begin{equation}
\label{eq:splitJ}
J = J_1 \cup J_2 \cup (J \setminus (J_1 \cup J_2),
\end{equation}
with $J_1 :=\{j \in J\ ;\ B(x_j , 10\varepsilon') \subset B(0,2-8 c_0)\ \text{and}\ B(0,r') \cap B(x_j, 10 \varepsilon') = \emptyset \}$, $J_2 :=\{j \in J\ ;\ B(x_j , 10\varepsilon') \subset B(0,2-8 c_0)\ \text{and}\ B(0,r') \cap B(x_j, 10 \varepsilon') \neq \emptyset \}$. To simplify the following proof, we will assume that $r/2 < \varepsilon$, then one can assume $r' <\varepsilon'$ because of \eqref{eq:defr'} so we have
\begin{equation}
\label{eq:assumptionr'}
B(0,r') \cap B(x_j, 10 \varepsilon') = \emptyset\qquad \forall j \in J,
\end{equation}
then $J_2 = \emptyset$. The other case, i.e. $r/2 > \varepsilon$ leading to $r' = cr$ by \eqref{eq:defr'} is a straightforward adaptation of the next arguments by using that $B(x_j, 10 \varepsilon') \subset B(0,(3/2) r')$ for every $j \in J_2$, recalling the form of $\varepsilon$, therefore all the terms coming from perforated disks in $J_2$ would be put in the observation term.\\

Let us then introduce $\eta$ a cut-off function such that
\begin{align}
\eta(x) &= 1 \qquad \forall r' \leq |x| \leq 2-8 c_0,\\
 \eta(x) &= 0\qquad \forall |x| \leq r'/2\ \text{or}\ 2-4 c_0 \leq |x|.
\end{align}
The goal of this part is to prove the following result.
\begin{lemma}
\label{lem:importantcarleman}
For every $p \geq 2$, there exists a positive constant $C \geq 1$ such that for every 
\begin{equation}
\label{eq:sfunctionvarepsilon}
s \geq C \varepsilon^{-1} \log(C \varepsilon^{-1}),
\end{equation}
the following estimate holds
\begin{align}
&\norme{\eta \zeta e^{\psi_s}}_{L^2(B_2)}^2 + \norme{\eta (\partial_{\overline{z}}\zeta) e^{\psi_s}}_{L^2(B_2)}^2+ \norme{\eta (\partial_{\overline{z}}\zeta) e^{\psi_s}}_{L^p(B_2)}^2\notag\\
& \leq C \Bigg( r'^{-2} \norme{\zeta e^{\psi_s}}_{L^2(C(0,r'/2,r'))}^2
+  \exp\left(C \varepsilon^{-1} \right) \exp(2\psi_s(2-16 c_0)) \|u\|_{L^{\infty}(B_2)}^2\Bigg).\label{eq:carlemanestimatenoboundary}
\end{align}
\end{lemma}

\begin{proof}
For the proof, to simplify the notations, we will denote
\begin{equation}
\label{eq:defB}
\mathcal B = C\exp\left(C \varepsilon^{-1} \right) \exp(2\psi_s(2-16 c_0)) \|u\|_{L^{\infty}(B_2)}^2.
\end{equation}
The proof is then divided into several steps.\medskip

\textbf{Step 1: Application of the Carleman estimate.} Let us apply the Carleman estimate \eqref{eq:Carlemany} to $y=\eta \zeta$. We obtain
\begin{multline}
\label{eq:applicationcarleman}
 \int_{B_2} |y|^2  e^{2\psi_s(z)} dz  \leq C \int_{B_{2}} \eta^2 |\partial_{\overline{z}} \zeta|^2 e^{2\psi_s(z)} dz \\
  + C (r')^{-2} \int_{r'/2 \leq |z| \leq r'} |\zeta|^2 e^{2\psi_s(z)} dz + C  \int_{2-8 c_0\leq |z| \leq 2} |\zeta|^2 e^{2 \psi_s(z)} dz.
\end{multline}
Let us first estimate the first right hand side term in \eqref{eq:applicationcarleman}. We have from \eqref{eq:equationzeta} the following estimate
\begin{align}
\norme{\eta (\partial_{\overline{z}}\zeta) e^{\psi_s}}_{L^2(B_2)}^2
&  \leq \norme{\eta e^{\psi_s} \exp(-  \beta ) [(\partial_{\overline{z}} \chi) h  + \tilde{E}_h] }_{L^2(B_2)}^2\notag\\
&\leq C  \norme{\eta e^{\psi_s} \exp(-  \beta ) |\nabla \chi | h}_{L^2(B_2)}^2
 + C  \norme{\eta e^{\psi_s} \exp(-  \beta ) \tilde{E}_h }_{L^2(B_2)}^2.\label{eq:decoupagerighthandside}
\end{align}

\textbf{Step 2: Absorption of the local term.} The first term on the right hand side of \eqref{eq:decoupagerighthandside} is called the local term and is estimated in the next paragraph.

By recalling the definition of $\chi$ in \eqref{eq:defyinfunctionh}, we have
\begin{equation}
\label{eq:estimationchi}
|\nabla \chi| \leq C \varepsilon^{-1}.
\end{equation}
So we get by \eqref{eq:splitJ}, the bound on $\beta$ i.e. \eqref{eq:boundbeta}, the choice of $\varepsilon$ in \eqref{eq:epsPetitPoincareGeneralFinStep2} and the definition of $h$ in \Cref{lem:harmonich}, 
\begin{align}
&\norme{\eta e^{\psi_s} \exp(-  \beta ) |\nabla \chi | h}_{L^2(B_2)}^2\\
& \leq C \varepsilon^{-2} \sum_{j \in J_1}  \norme{e^{\psi_s}  \exp(-  \beta ) h}_{L^2(B(x_j',4 \varepsilon'))}^2\notag
+ C \varepsilon^{-2} \norme{ e^{\psi_s} \exp(-  \beta ) h}_{L^2(C(0,2-16 c_0, 2-4 c_0))}^2\notag\\
&\leq C \varepsilon^{-2} \sum_{j \in J_1}  \norme{e^{\psi_s(z)}  \exp(-  \beta ) h}_{L^2(B(x_j',4 \varepsilon'))}^2 + \mathcal{B}.\label{eq:estimatelocal1}
\end{align}
Then, we use  the estimates \eqref{eq:harnacklocal}, \eqref{eq:expTwLlocal} and \eqref{eq:carlemanweightlocal} of \Cref{lem:estimationlocCarleman} to easily get
\begin{multline}
\norme{e^{\psi_s(z')} \exp(- \beta(z')) h (z')}_{L^2(B(x_j',4 \varepsilon'))}^2 \\
\leq C \exp(-2 c s \varepsilon)  \norme{e^{\psi_s(z)} \exp(- \beta(z)) h (z)}_{L^2(C(x_j',6 \varepsilon',8 \varepsilon'))}^2,\label{eq:estimatelocal2}
\end{multline}
and from \eqref{eq:assumptionr'}, the definition of $\zeta$ in \eqref{eq:defzeta}, the definition of $\gamma$ in \eqref{eq:defgamma} and $|\gamma| \geq |Re(\gamma)| = |\chi h|$,
\begin{equation}
\norme{y(z) e^{s \psi_s(z)}}_{L^2(B_2)}^2 
\geq \sum_{j \in J_1} \norme{e^{\psi_s(z)} \exp(- \beta(z)) h (z)}_{L^2(C(x_j',6 \varepsilon',10 \varepsilon'))}^2, \label{eq:estimatelocal3}
\end{equation}
We then deduce from \eqref{eq:estimatelocal1}, \eqref{eq:estimatelocal2} and \eqref{eq:estimatelocal3} and by taking $s$ as in \eqref{eq:sfunctionvarepsilon},
\begin{align}
\norme{\eta e^{\psi_s} \exp(-  \beta ) |\nabla \chi | h}_{L^2(B_2)}^2& \leq C \varepsilon^{-2} \exp(-2 c s \varepsilon) \norme{y e^{\psi_s}}_{L^2(B_2)}^2 +  \mathcal{B}\notag\\
& \leq \frac{1}{4} \norme{y e^{\psi_s}}_{L^2(B_2)}^2 + \mathcal{B}.\label{eq:decoupagerighthandside2}
\end{align}

\begin{figure}
\centering
\begin{tikzpicture}[scale=0.27]
\draw[fill=magenta] (0,0) circle (4);
\draw (0,0) node{\scriptsize{$B(x_j', 4\varepsilon')$}};
\draw [fill=cyan,even odd rule] (0,0) circle[radius=10cm] circle[radius=6cm];
\draw (0,7.5) node{\scriptsize{$C(x_j', 6\varepsilon', 8 \varepsilon')$}};
\draw (-20,0) node[below]{$0$};
\draw (-20,0) node {$\times$};
\end{tikzpicture}
\caption{The neighbourhood of a disk $D_j'$ of the perforated domain.}
    \label{fig:Weight}
\end{figure}
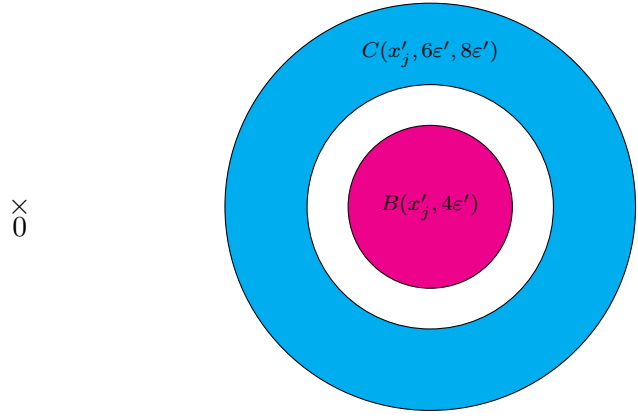

\textbf{Step 3: Absorption of the non-local term.} The second term in \eqref{eq:decoupagerighthandside} is called the non local term and is estimated in the next paragraph.

By recalling the form of $\tilde{E}_h$ in \eqref{eq:ehtilde}, by using that the disks $B(x_j', 4 \varepsilon')$ are $12 \varepsilon'$-separated from $0$, hence $E_h$ is supported in $C(0,\varepsilon',2)$, we first obtain that 
\begin{equation}
\norme{\eta e^{\psi_s} \exp(-  \beta ) \tilde{E}_h }_{L^2(B_2)}^2 
\leq C \varepsilon^{-2} \norme{\eta e^{\psi_s} \exp(-  \beta ) {E}_h }_{L^2(B_2)}^2.
\end{equation}
Then by using the definition of $E_h$ in \eqref{eq:defEh}, Hölder inequality, the bound on $\nabla \chi$ in \eqref{eq:estimationchi}, $L^p$ estimate on $\nabla h$ at scale $\varepsilon$ from \eqref{eq:estimationgradientcutoff}, $L^p$ estimate on $\nabla h$ from \eqref{eq:estimationgradientcutoffboundary}, we then obtain
\begin{align}
&\norme{\eta e^{\psi_s} \exp(-  \beta ) \tilde{E}_h }_{L^2(B_2)}^2\notag \\
&\leq C \varepsilon^{-4}\sum_{j \in J_1} \norme{e^{\psi_s} \exp(-  \beta )}_{L^4(B_{0,2}(x_j', 4 \varepsilon'))}^2
\cdot \left(\norme{\nabla h }_{L^4(B(x_j', 4 \varepsilon'))}^2 + \norme{\hat{W}}_{\infty}^2 \norme{h }_{L^4(B(x_j', 4 \varepsilon'))}^2\right)\notag\\
&\qquad+ C \varepsilon^{-4}  \norme{ e^{\psi_s} \exp(-  \beta )}_{L^4(C(0,2-16 c_0 , 2-4 c_0))}^2\notag \\
&\qquad\qquad \cdot \left(\norme{\nabla h }_{L^4(C(0,2-16 c_0 , 2-4 c_0))}^2 + \norme{\hat{W}}_{\infty}^2 \norme{h }_{L^4(C(0,2-16 c_0 , 2-4 c_0))}^2\right)\notag\\
&\leq C \varepsilon^{-6}\sum_{j \in J_1} \norme{e^{\psi_s} \exp(-  \beta )}_{L^4(B_{0,2}(x_j', 4 \varepsilon'))}^2\norme{ h }_{L^{\infty}(B(x_j', 8 \varepsilon'))}^2 + \mathcal{B},\label{eq:nonlocaldecoupling}
\end{align}
where 
$$ B_{0,2}(x_j', 4 \varepsilon') = \{z' = z + t \frac{z}{|z|}\ ;\ z \in B(x_j',4 \varepsilon'),\ 0 \leq t \leq 2\} \cap B(0,2).$$
Note that $B_{0,2}(x_j', 4 \varepsilon')$ is like a cone that we decide to represent several points in \Cref{fig:WeightCone}.

We only focus on the first right hand side term of \eqref{eq:nonlocaldecoupling}. We split 
\begin{equation}
B_{0,2}(x_j', 4 \varepsilon') = B_{0,\varepsilon'}(x_j', 4\varepsilon') \cup B_{\varepsilon',2}(x_j', 4 \varepsilon').
\end{equation}
So we have
\begin{multline}
 \norme{e^{\psi_s} \exp(-  \beta )}_{L^4(B_{0,2}(x_j', 4 \varepsilon'))}^2
\leq  \norme{e^{\psi_s} \exp(-  \beta )}_{L^4(B_{0,\varepsilon'}(x_j', 4 \varepsilon'))}^2 \\
+ \norme{e^{\psi_s} \exp(-  \beta )}_{L^4(B_{\varepsilon',2}(x_j', 4 \varepsilon'))}^2.
\end{multline}
By using $s \geq C \varepsilon^{-1}$ because of \eqref{eq:sfunctionvarepsilon} and $\varepsilon^{Q} \|\hat{W}\|_{\infty} \leq c$ because of \eqref{eq:epsPetitPoincareGeneralFinStep2}, then for $\varepsilon \leq t \leq 2$, we have
\begin{multline}
 \exp (C \|\hat{W}\|_{\infty} t^Q) \exp(-cst) \leq \exp (C \|\hat{W}\|_{\infty} \varepsilon^{Q-1} t) \exp(- C \varepsilon^{-1} t) \\
 \leq \exp(c \varepsilon^{-1} t)\exp(- C \varepsilon^{-1} t) \leq C.
\end{multline}
So from \Cref{lem:estimationnonlocCarleman}, \eqref{eq:estimationbetanonlocal} and \eqref{eq:estimationpoidsnonlocal}, we get from the last two estimates
\begin{equation}
\norme{e^{\psi_s} \exp(- \beta )}_{L^4(B_{0,2}(x_j', 4 \varepsilon'))}^2 
\leq C \norme{e^{\psi_s} \exp(-  \beta )}_{L^4(B(x_j', 5 \varepsilon'))}^2,
\end{equation}
and therefore we have
\begin{multline}
\label{eq:nonlocalajoutdetails}
\varepsilon^{-6} \norme{e^{\psi_s} \exp(-  \beta )}_{L^4(B_{0,2}(x_j', 4 \varepsilon'))}^2\norme{ h }_{L^{\infty}(B(x_j', 8 \varepsilon'))}^2 \\
\leq C \varepsilon^{-6}  \norme{e^{\psi_s} \exp(-  \beta )}_{L^4(B(x_j', 5 \varepsilon'))}^2 \norme{ h }_{L^{\infty}(B(x_j', 8 \varepsilon'))}^2.
\end{multline}
Note that the right hand side term in \eqref{eq:nonlocalajoutdetails} looks like a local term, therefore from an easy adaptation of \Cref{lem:estimationlocCarleman},
\begin{multline}
\varepsilon^{-6} \norme{e^{\psi_s} \exp(-  \beta )}_{L^4(B_{0,2}(x_j', 4 \varepsilon'))}^2 \norme{ h }_{L^{\infty}(B(x_j', 8 \varepsilon'))}^2
\leq C \varepsilon^{-6} \norme{e^{\psi_s} \exp(-  \beta ) h}_{L^4(B_{0,2}(x_j', 4 \varepsilon'))}^2\\
\leq C \varepsilon^{-5} \norme{e^{\psi_s} \exp(-  \beta ) h}_{L^{\infty}(B_{0,2}(x_j', 4 \varepsilon'))}^2
\leq C \varepsilon^{-7} \exp(-2 c s \varepsilon) \norme{e^{\psi_s} \exp(- \beta h}_{L^2(C(x_j',6 \varepsilon',8 \varepsilon'))}^2.
\end{multline}
Then, recalling \eqref{eq:estimatelocal3}, \eqref{eq:nonlocaldecoupling} and by taking $s$ as in \eqref{eq:sfunctionvarepsilon}, we have 
\begin{align}
\norme{\eta e^{\psi_s} \exp(-  \beta ) \tilde{E}_h }_{L^2(B_2)}^2 & \leq C \varepsilon^{-7} \exp(-2 c s \varepsilon) \norme{y e^{\psi_s}}_{L^2(B_2)}^2 +  \mathcal{B}\notag\\
& \leq \frac{1}{4} \norme{y e^{\psi_s}}_{L^2(B_2)}^2 + \mathcal{B}.\label{eq:decoupagerighthandside2nonlocal}
\end{align}

\begin{figure}
\centering
\begin{tikzpicture}[scale=0.47]
\draw[fill=magenta] (0,0) circle (2);
\draw (0,0) node{\scriptsize{$B(x_j', 4\varepsilon')$}};
\draw [fill=cyan,even odd rule] (0,0) circle[radius=5cm] circle[radius=3cm];
\draw (0,3.60) node{\scriptsize{$C(x_j', 6\varepsilon', 10 \varepsilon')$}};
\draw (-15,0) node[below]{$0$};
\draw (-15,0) node {$\times$};
\coordinate (0) at (-15,0);
\coordinate (N) at (15,3);
\draw  (15,3) node[below]{$z_1$};
\draw  (15,3) node {$\times$};
\draw (0)--(N);
\coordinate (P) at (12,2);
\draw  (12,2) node[below]{$z_2$};
\draw  (12,2) node {$\times$};
\draw (0)--(P);
\coordinate (Q) at (14,3.8);
\draw  (14,3.8) node[below]{$z_3$};
\draw  (14,3.8) node {$\times$};
\draw (0)--(Q);
\coordinate (R) at (10,-2);
\draw  (10,-2) node[below]{$z_4$};
\draw  (10,-2) node {$\times$};
\draw (0)--(R);
\coordinate (S) at (4,-0.8);
\draw  (4,-0.8) node[below]{$z_5$};
\draw  (4,-0.8) node {$\times$};
\draw (0)--(S);
\end{tikzpicture}
\caption{The representation of several points $z_i \in B_{0,2}(x_j', 4 \varepsilon')$.}
    \label{fig:WeightCone}
\end{figure}
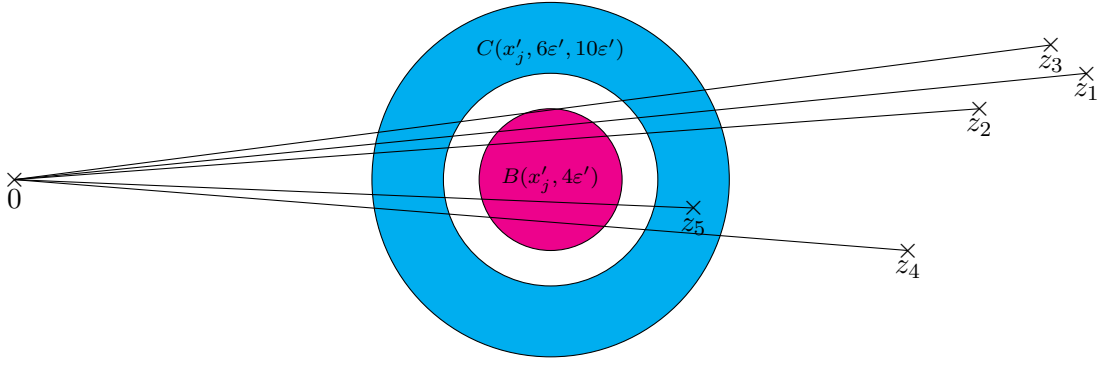

\textbf{Step 4: Treatment of the boundary term.} The last term in the right hand side of \eqref{eq:applicationcarleman} is called a boundary term and can be estimated by using the definition of $\zeta$ in \eqref{eq:defzeta}, the definition of $\gamma$ in \eqref{eq:defgamma}, the definition of $h$ in \eqref{eq:defh}, the definition of $\tilde{h}$ in \eqref{eq:defvrhotheta}, the bound on $\beta$ i.e. \eqref{eq:boundbeta} and standard elliptic estimates applied to \eqref{eq:equationh} to get
\begin{equation}
\label{eq:boundaryzeta}
\int_{2-8 c_0 \leq |z| \leq 2} |\zeta|^2 e^{2 \psi_s(z)} dz \leq \mathcal{B}.
\end{equation}

\textbf{Step 5: Conclusion.} 
By gathering \eqref{eq:applicationcarleman}, \eqref{eq:decoupagerighthandside}, \eqref{eq:decoupagerighthandside2}, \eqref{eq:decoupagerighthandside2nonlocal} and \eqref{eq:boundaryzeta}, we obtain the following Carleman estimate
\begin{equation}
\norme{\eta \zeta e^{\psi_s}}_{L^2(B_2)}^2 + \norme{\eta (\partial_{\overline{z}} \zeta) e^{\psi_s}}_{L^2(B_2)}^2 \leq C \Bigg( r'^{-2} \norme{\zeta e^{\psi_s}}_{L^2(C(0,r'/2,r'))}^2
+ \mathcal{B} \Bigg).
\end{equation}
Moreover, from an easy adaptation of Step 2 and Step 3, it is straightforward to obtain under the condition on $s$, i.e. \eqref{eq:sfunctionvarepsilon},
\begin{equation}
 \norme{\eta (\partial_{\overline{z}} \zeta) e^{\psi_s}}_{L^p(B_2)}^2 \leq \frac{1}{2} \norme{\eta \zeta e^{\psi_s}}_{L^2(B_2)}^2 + \mathcal{B} .
\end{equation}
Hence, we obtain the expected Carleman estimate \eqref{eq:carlemanestimatenoboundary} from the last two estimates then the conclusion of the proof.
\end{proof}

\subsection{From the Carleman estimate to the quantitative unique continuation}

The goal of this last part is to finish the proof of \eqref{eq:observabilityu}.

For the next and without loss of generality, we assume that 
\begin{equation}
\label{eq:estimateLxmax}
|L(x_{\max})| \geq r'.
\end{equation}
Indeed, if this is not the case this means that $x_{\max}$ belongs to $B(0,r/2)$, because $B(0,2r') \subset L(B(0,r/2))$ by the last point of \Cref{lem:imagequasiconformal}, and the observability estimate \eqref{eq:observabilityu} is trivial.

\begin{proof}[Proof of the estimate \eqref{eq:observabilityu}]
From the Carleman estimate \eqref{eq:carlemanestimatenoboundary}, we obtain
\begin{multline}
e^{2\psi_s(2-24 c_0)} \int_{B(0,2-24 c_0)}  |\eta|^2 [|\zeta |^2 + |\partial_{\overline{z}}\zeta |^2 ] dz\\
\leq C \Bigg( r'^{-2} \norme{\zeta e^{\psi_s}}_{L^2(C(0,r'/2,r'))}^2
+   \exp\left(C \varepsilon^{-1} \right) \exp(2\psi_s(2-16 c_0)) \|u\|_{L^{\infty}(B_2)}^2\Bigg).
\end{multline}
By using regularity estimate on non-homogeneous $\partial_{\overline{z}}$ equation, see for instance \cite[Theorem 5.4.3]{AIM09}, we can therefore deduce an estimate of the form
\begin{multline}
e^{2\psi_s(2-24 c_0)} \norme{\eta \zeta}_{W^{1,2}(B(0,2-28 c_0))}^2\\
 \leq C \Bigg( r'^{-2} \norme{\zeta e^{\psi_s}}_{L^2(C(0,r'/2,r'))}^2
+   \exp\left(C \varepsilon^{-1} \right) \exp(2\psi_s(2-16 c_0)) \|u\|_{L^{\infty}(B_2)}^2\Bigg).
\end{multline}
So by Sobolev embedding, for $2 \leq p<  \infty$,
\begin{multline}
\label{eq:toL2toLpCarleman}
e^{2\psi_s(2-24 c_0)} \norme{\eta \zeta}_{L^p(B(0,2-28 c_0))}^2\\
\leq C \Bigg( r'^{-2} \norme{\zeta e^{\psi_s}}_{L^2(C(0,r'/2,r'))}^2
+   \exp\left(C \varepsilon^{-1} \right) \exp(2\psi_s(2-16 c_0)) \|u\|_{L^{\infty}(B_2)}^2\Bigg).
\end{multline}
So by using the third left hand side term in \eqref{eq:carlemanestimatenoboundary} and \eqref{eq:toL2toLpCarleman}, together with regularity estimate on non-homogeneous $\partial_{\overline{z}}$ equation, see for instance \cite[Theorem 5.4.3]{AIM09}, we have for $2 \leq p < \infty$,
\begin{multline}
e^{2\psi_s(2-24 c_0)}  \norme{\eta \zeta}_{W^{1,p}(B(0,2-30 c_0))}^2\\
 \leq C \Bigg( r'^{-2} \norme{\zeta e^{\psi_s}}_{L^p(C(0,r'/2,r'))}^2
+   \exp\left(C \varepsilon^{-1} \right) \exp(2\psi_s(2-16 c_0)) \|u\|_{L^{\infty}(B_2)}^2\Bigg).
\end{multline}
So by using Sobolev embedding, taking $p \in (2, \infty)$, we finally get
\begin{multline}
\label{eq:intermediatefincarleman}
e^{2\psi_s(2-24 c_0)} \norme{\eta \zeta}_{L^{\infty}(B(0,2-30 c_0))}^2\\
 \leq C \Bigg( r'^{-2} \norme{\zeta e^{\psi_s}}_{L^{\infty}(C(0,r'/2,r'))}^2
+  \exp\left(C \varepsilon^{-1} \right) \exp(2\psi_s(2-16 c_0)) \|u\|_{L^{\infty}(B_2)}^2\Bigg).
\end{multline}
Recalling the definition of $c_0$ in \eqref{eq:defc0}, we have that $B(0,2-2^{-5}) \subset B(0,2-30 c_0)$. Moreover from our assumptions $L(x_{\max}) \in B(0,2-2^{-5})$ because $x_{\max} \in B(0,1)$ and we have the Lipschitz estimate on $L$, i.e. \eqref{eq:LipschitzL}. By conjugating \eqref{eq:intermediatefincarleman} with \eqref{eq:estimateLxmax}, by using the definition of $\zeta$ in \eqref{eq:defzeta}, the definition of $\gamma$ in \eqref{eq:defgamma}, the definition of $h$ in \eqref{eq:defh}, the definition of $\tilde{h}$ in \eqref{eq:defvrhotheta}, the bound on $\beta$ i.e. \eqref{eq:boundbeta}, we then deduce
\begin{multline}
\exp(-C \|\hat{W}\|_{\infty}) e^{2\psi_s(2-24 c_0)}|h (L(x_{\max}))|^2\\
 \leq C \Bigg( r'^{-2} \norme{\zeta e^{\psi_s}}_{L^{\infty}(C(0,r'/2,r'))}^2
+  \exp\left(C \varepsilon^{-1} \right) \exp(2\psi_s(2-16 c_0)) \|u\|_{L^{\infty}(B_2)}^2\Bigg),
\end{multline}
that translates into 
\begin{multline}
\label{eq:obsboundaryabsorb}
\exp(-C \|\hat{W}\|_{\infty})e^{2\psi_s(2-24 c_0)}| |u(x_{\max})|^2 \\
\leq C \Bigg( r'^{-2} \norme{\zeta e^{\psi_s}}_{L^{\infty}(C(0,r'/2,r'))}^2
+  \exp\left(C \varepsilon^{-1} \right) \exp(2\psi_s(2-16 c_0)) \|u\|_{L^{\infty}(B_2)}^2\Bigg).
\end{multline}
Then from the assumption \eqref{eq:ucontrolledneartheboundary},
we then have
\begin{multline}
\label{eq:obsboundaryabsorbBis}
\exp(-C \|\hat{W}\|_{\infty}) e^{2\psi_s(2-24 c_0)} |u(x_{\max})|^2 \\
\leq C \Bigg( r'^{-2} \norme{\zeta e^{\psi_s}}_{L^{\infty}(C(0,r'/2,r'))}^2
+  \exp\left(C \varepsilon^{-1} \right) \exp(2\psi_s(2-16 c_0)) \exp (2K) |u(x_{\max})|^2\Bigg).
\end{multline}
As a consequence, from \eqref{eq:epsPetitPoincareGeneralFinStep2}, by taking $s$ such that 
\begin{equation}
\label{eq:sfunctionvarepsilonK}
s  \geq C  \varepsilon^{-1} \log(C \varepsilon^{-1}) + C K + C \|\hat{W}\|_{\infty},
\end{equation}
we obtain from \eqref{eq:obsboundaryabsorbBis}
\begin{equation}
\exp(-C \|\hat{W}\|_{\infty}) e^{2\psi_s(2-24 c_0)}|u(x_{\max})|^2 \leq C  r'^{-2} \norme{\zeta e^{\psi_s}}_{L^{\infty}(C(0,r'/2,r'))}^2.
\end{equation}
Finally, from the definition of $\zeta$ in \eqref{eq:defzeta}, the bound on $\beta$ i.e. \eqref{eq:boundbeta}, the definition of $\gamma$ in \eqref{eq:defgamma}, the definition of $\tilde{h}$ in \eqref{eq:defvrhotheta}, standard elliptic estimates applied to $h$, the form of $r'$ in function of $r$ given in \eqref{eq:defr'}, one can deduce
\begin{equation}
|u(x_{\max})|^2 \leq C  \exp(-C s \log(r)) \norme{h}_{L^{\infty}(B(0,2r'))}^2 \leq C e^{ - C s  \log(r)}  \norme{u}_{L^{\infty}(B(0,r))}^2,
\end{equation}
then the expected estimate \eqref{eq:observabilityu} from \eqref{eq:sfunctionvarepsilonK}, \eqref{eq:defW} and \eqref{eq:epsPetitPoincareGeneralFinStep2}.
\end{proof}

\section{Final comments}
\label{sec:finalcomments}

\subsection{Perspectives and open questions}

The following remarks concerning \Cref{thm:landislocal}, that implies all the other main results according to \eqref{eq:implicationtheorem}, are in order.
\begin{enumerate}
\item A natural important question is the sharpness of \eqref{eq:observabilityu} in function of
$\|W_1\|_{\infty}, \|W_2\|_{\infty}, \|V\|_{\infty}$ in the real-valued case. Note that the complex-valued case is by now quite well-understood, see \cite{Dav14}. This is definitely not sharp in function of $\|V\|_{\infty}$ because \cite{LMNN20} obtains, when $W_1 = W_2 = 0$, 
$$\norme{u}_{L^{\infty}(B_r)} \geq r^{C \left( \|V\|_{\infty}^{1/2} \log_{+}^{1/2}(\|V\|_{\infty})\right)+ C K + C} \norme{u}_{L^{\infty}(B_2)}\quad  \forall r \in (0,1/2).$$
This is also definitely not sharp in function of $\|W_1\|_{\infty}, \|W_2\|_{\infty}$ because \cite{KSW15} obtains when $W_2=0$,  $V \geq 0$, respectively when $W_1 = 0$, $V \geq 0$,
$$\norme{u}_{L^{\infty}(B_r)} \geq r^{C \left( \|W_1\|_{\infty} + \|V\|_{\infty}^{1/2} \right)+ C K + C} \norme{u}_{L^{\infty}(B_2)}\quad  \forall r \in (0,1/2),$$
respectively
$$\norme{u}_{L^{\infty}(B_r)} \geq r^{C \left( \|W_2\|_{\infty} + \|V\|_{\infty}^{1/2}\right)+ C K + C} \norme{u}_{L^{\infty}(B_2)}\quad  \forall r \in (0,1/2).$$
But all these examples are not considering the general equation \eqref{eq:equationuLandisquantitativelocal}. In order to improve \eqref{eq:observabilityu} at least in function of $\|W_1\|_{\infty}, \|W_2\|_{\infty}$ a possible strategy is to investigate where we are losing a factor $\|W_1\|_{\infty}^{\delta}$, $\|W_2\|_{\infty}^{\delta}$. It already appears in Step 1, because of the weak quantitative maximum principles with $W^{-1, \infty}$ source term $-\Delta \Phi = \nabla \cdot g$ in a domain with a small Poincaré constant of size $C \varepsilon$, see \Cref{lem:varphiTildeDivergenceSource}. It would be interesting to see if one can improve the $L^{\infty}$-bound on $\Phi$ into $\|\Phi\|_{\infty} \leq C \varepsilon \|g\|_{\infty}$. If such a bound was true, this would lead to the replacement of $\varepsilon$ in Step 1, i.e. one could modify \eqref{eq:firstvarepsilon} into
\begin{equation*}
\label{eq:firstvarepsilonconj}
\varepsilon \leq c +  c \|W_1\|_{\infty}^{-1} + c \|W_2\|_{\infty}^{-1} +  c \|V\|_{\infty}^{-1/2}.
\end{equation*}
In Step 2, by the quasiconformal transformation, \eqref{eq:secondvarepsilon} would become
\begin{equation*}
\varepsilon \leq c \|W_2\|_{\infty}^{-1}\log^{-1} (\|W_2\|_{\infty}) + c \|V\|_{\infty}^{-1/2} \log^{-1/2}(\|V\|_{\infty}).
\end{equation*}
This logarithm loss is probably optimal according to \cite{LMNN-Personal}. However, there is still a problem because of the presence of $\partial_{\overline{z}} \overline{L^{-1}}$ in the definition of $\tilde{W}$ in \eqref{eq:deftildeW1}, \eqref{eq:deftildeW2}. We are only able to prove that $\partial_{\overline{z}} \overline{L^{-1}} \in L_{\mathrm{loc}}^p(B_2)$ for every $1 \leq p < \infty$ by using Cacciopoli's estimate. In particular, we do not know if $\partial_{\overline{z}} \overline{L^{-1}} \in L_{\mathrm{loc}}^{\infty}(B_2)$ and if one can obtain $$\|\tilde{W}\|_{L^{\infty}(B_{2-c})} \leq C \|W_1\|_{\infty} + C \|W_2\|_{\infty}.$$ 
If such a bound was true, one can then improve the Hölder's estimate \eqref{eq:HolderTSpecific} on $\beta$ into
\begin{equation}
|\beta (z_1) - \beta(z_2)|  \leq \left(C \|W_1\|_{\infty} + C |W_2\|_{\infty}\right) |z_1-z_2| \log\left( \frac{C}{|z_1-z_2|}\right),
\end{equation}
from \cite[Chapter 1, Paragraph 6]{Vek62}. By taking then
\begin{equation}
\label{eq:thirdvarepsilon}
\varepsilon \leq c \|W_1\|_{\infty}^{-1}\log^{-1} (\|W_1\|_{\infty}) + c \|W_2\|_{\infty}^{-1}\log^{-1} (\|W_2\|_{\infty}) + c \|V\|_{\infty}^{-1/2} \log^{-1/2}(\|V\|_{\infty}).
\end{equation}
This would be sufficient for the absorption of the local and non local terms in the last step, by taking $s \geq C \varepsilon^{-1} \log( C \varepsilon^{-1})$. Finally, to withdraw the logarithm loss in the Carleman step, a possible strategy would be to use the strategy in \cite{DF90} to add in the left hand side of the Carleman estimate a term of the form
\begin{equation}
c \varepsilon^{-2} \sum_{j \in J} \norme{\eta \zeta e^{\psi_s(z)}}_{L^2(C(x_j', 6 \varepsilon', 8 \varepsilon'))}^2.
\end{equation}
\item A natural possible extension of our results is to consider qualitative and quantitative Landis conjecture for exterior domains, that is for instance does \Cref{Thm:Landis}  holds replacing the equation \eqref{eq:equationuLandis} satisfied by $u$ in $\R^2$ by the same equation but only in the exterior domain $\R^2 \setminus \overline{B_1}$? See for instance \cite[Section 5]{KSW15}.
\item Another possible extension would be to consider the more general elliptic equation
\begin{equation}
\label{eq:equationuLandisquantitativelocalverygeneral}
- \sum_{i,j=1}^2 \partial_i (a_{ij} \partial_j u) - \nabla \cdot ( W_1 u )  +W_2 \cdot \nabla u  + V u = 0\ \text{in}\ B_2,
\end{equation}
where $a_{ij} \in L^{\infty}(B_2;\R^{2 \times 2})$ satisfying some ellipticity condition, $W_1, W_2 \in L^p(B_2;\R^2)$, $V \in L^{q}(B_2:\R)$ with $p>d$, $q>d/2$. The general situation is open even if interesting partial results are given in \cite{DKW17}, \cite{DW20} or very recently for growing potentials in \cite{Dav23}.
\item In the spirit of the methodology developed in \cite{ELB22}, it may also be natural to consider equations with source term, that is
\begin{equation}
\label{eq:equationuLandisquantitativelocalFinalCommentsSource}
-\Delta u - \nabla \cdot ( W_1 u )  +W_2 \cdot \nabla u  + V u = f\ \text{in}\ B_2.
\end{equation}
It seems natural to conjecture that under the assumption \eqref{eq:ucontrolledneartheboundary}, the following estimate holds for every $r \in (0,1/2)$,
\begin{equation}
\label{eq:observabilityusource}
\norme{u}_{L^{\infty}(B_r)} + \norme{f}_{L^{\infty}(B_2)} \geq r^{C \left(\|W_1\|_{\infty}^{1 + \delta} + \|W_2\|_{\infty}^{1 + \delta} + \|V\|_{\infty}^{1/2} \log_{+}^{3/2}(\|V\|_{\infty})\right)+ C K  + C} \norme{u}_{L^{\infty}(B_2)}.
\end{equation}
It is worth mentioning that such an estimate would lead to applications in control theory of linear and semi-linear elliptic equations.
\item Last but not least, always in the spirit of \cite{ELB22}, the treatment of elliptic equations \eqref{eq:equationuLandisquantitativelocalFinalCommentsSource} completed with Dirichlet boundary conditions for instance, is an interesting open question. Here, one of the difficulty is due to the fact that the boundary conditions are not preserved in the Step 2 of the proof because we are considering the variable $\gamma$ defined in \eqref{eq:defgamma}.
\end{enumerate}

\subsection{A Carleman estimate in a bounded open set with small Poincaré constant}

The goal of this part is to present a specific two-dimensional Carleman estimate in a bounded open set with small Poincaré constant. This type of estimate seems to be new but unfortunately we have not found applications of it. We hope that it can be useful for the reader in another context.

\begin{lemma}
\label{lem:carlemansmallpoincare}
For every $\varepsilon>0$, $C' \geq 1$, there exists $C>0$, independent of $\varepsilon$, such that for every bounded open set $\Omega \subset \R^2$ with $C_P(\Omega)^2 \leq (C')^2 \varepsilon^2$, for every $s \geq 1$, for every $\varphi \in C^{\infty}(\Omega;\R)$ such that $-\Delta \varphi \geq s \geq 1$, for every $u \in H_0^2(\Omega;\R)$, we have
\begin{equation}
\varepsilon^{-4} \int_{\Omega} e^{-\varphi} |u|^2 + s^2 \int_{\Omega}   e^{-\varphi} |u|^2 + \varepsilon^{-2} \int_{\Omega} e^{-\varphi} |\nabla u|^2 + s \int_{\Omega}  e^{-\varphi} |\nabla u|^2
 \leq C \int_{\Omega}  e^{-\varphi} |\Delta u|^2.
\end{equation}
\end{lemma}
The idea behind this was to exploit this Carleman estimate in the domain $\Omega_{\varepsilon}$, defined in \eqref{Omega-Eps} that has a small Poincaré constant. But even the solution $u$ to $-\Delta u + Vu=0$ belongs to $H_0^1(\Omega_{\varepsilon})$, we do not have $u \in H_0^2(\Omega_{\varepsilon})$ that prevents the application of such an inequality.

\begin{proof}
We can assume that $u \in C_c^{\infty}(\Omega;\R)$ by density.

By \Cref{lem:varphiTilde}, one can construct $\Phi \in H_0^1(\Omega)$ such that $-\Delta \Phi =  \varepsilon^{-2}     \text{ in } \Omega$ and $\Phi$ satisfies $\norme{\Phi}_{\infty} \leq C$. Then we directly apply \cite[Proposition 2.2]{DF90} with the weight $e^{-\varphi + \Phi}$ to get
\begin{equation}
\varepsilon^{-2} \int_{\Omega} e^{- \varphi+\Phi} |u|^2 + s  \int_{\Omega} e^{- \varphi+\Phi} |u|^2 \leq C \int_{\Omega}  e^{- \varphi+\Phi} |\overline{\partial} u|^2.
\end{equation}
Now, because $u$ is real-valued, $|\overline{\partial} u| = |\partial u| = |\nabla u|$, so one can apply  \cite[Proposition 2.2]{DF90} to $\partial{u}$,
\begin{equation}
\varepsilon^{-2} \int_{\Omega} e^{- \varphi+\Phi} |\partial u|^2 + s  \int_{\Omega} e^{- \varphi+\Phi} |\partial u|^2 \leq C \int_{\Omega}  e^{- \varphi+\Phi} |\overline{\partial} \partial u|^2 = C \int_{\Omega}  e^{- \varphi+\Phi} |\Delta u|^2
\end{equation}
The combination of the last two estimates and the fact that $\Phi$ is uniformly bounded give the result.
\end{proof}

\appendix

\section{Estimates for elliptic equations with lower order terms}

We have the following local $W^{1,p}$ estimate for second order elliptic equations with divergence drift term.
\begin{lemma}
\label{lem:gradientdivtermscale1}
For every $p \in (1, \infty)$, there exists $C>0$ such that for every $W \in L^{p}(B_2)$, for every $u \in H_{\mathrm{loc}}^1(B_2) \cap L^{\infty}(B_2)$ satisfying
\begin{equation}
- \Delta u - \nabla \cdot  (W u) = 0\ \text{in}\ B_2,
\end{equation}
then,
\begin{equation}
\norme{\nabla u}_{L^p(B_1)} \leq C \norme{W}_{L^p(B_2)} \norme{u}_{L^{\infty}(B_2)} + C \norme{u}_{L^{\infty}(B_2)}.
\end{equation}
\end{lemma}
\begin{proof}
This is a direct application of \cite[Theorem 2]{Mey63}, stating that if $u$ is a weak solution to 
$$ - \Delta u = \nabla \cdot g\ \text{in}\ B_2,$$
then we have the local estimate
\begin{equation}
\norme{\nabla u}_{L^p(B_1)} \leq C   \norme{g}_{L^p(B_2)} + C \norme{u}_{L^p(B_2)}.
\end{equation}
This concludes the proof by taking $g = Wu$.
\end{proof}

We have the following local $W^{1,p}$ estimate for second order elliptic equations with divergence drift term at scale $\varepsilon$.
\begin{lemma}
\label{lem:gradientdivterm}
For every $p \in (1, \infty)$, $\varepsilon>0$, there exist $c>0$, $C>0$ independent of $\varepsilon$ such that for every $W \in L^{p}(B_{2\varepsilon})$, satisfying
\begin{equation}
\varepsilon + \varepsilon\|W\|_{L^p(B_{2 \varepsilon})}  \leq c,
\end{equation}
and 
for every $u \in H_{\mathrm{loc}}^1(B_2) \cap L^{\infty}(B_2)$ satisfying
\begin{equation}
- \Delta u - \nabla \cdot  (W u) = 0\ \text{in}\ B_{2\varepsilon},
\end{equation}
then we have
\begin{equation}
\norme{\nabla u}_{L^p(B_{ \varepsilon})} \leq C \varepsilon^{-1} \norme{u}_{L^{\infty}(B_{2 \varepsilon})}.
\end{equation}
\end{lemma}
\begin{proof}
By \cite[Theorem 2]{Mey63} and a scaling argument, if
$$ - \Delta u =  \nabla \cdot g\ \text{in}\ B_{2 \varepsilon},$$
then we have the local estimate
\begin{equation}
\varepsilon \norme{\nabla u}_{L^p(B_{\varepsilon})} \leq  C \varepsilon \norme{g}_{L^p(B_{2 \varepsilon})} + C \norme{u}_{L^p(B_{2 \varepsilon})}.
\end{equation}
Taking $g = Wu$ leads to the conclusion.
\end{proof}

We now have the following Harnack's inequalities.
\begin{lemma}
\label{lem:harnackscaleeps}
For every $\varepsilon>0$, there exist $c>0$, $C>0$ independent of $\varepsilon$ such that for every $W_1, W_2, V \in L^{\infty}(B_{2\varepsilon})$ satisfying
\begin{equation}
\varepsilon + \varepsilon\|W_1\|_{\infty} + \varepsilon\|W_2\|_{\infty} + \varepsilon^2 \|V\|_{\infty} \leq c,
\end{equation}
and for every $u \in H_{\mathrm{loc}}^1(B_2) \cap L^{\infty}(B_2)$ satisfying
\begin{equation}
- \Delta u - \nabla \cdot  (W u) = 0\ \text{in}\ B_{2\varepsilon},
\end{equation}
if
\begin{equation}
u > 0\ \text{in}\ B_{2\varepsilon},
\end{equation}
then we have 
\begin{equation}
\sup_{B_{\varepsilon}} u \leq C \inf_{B_{\varepsilon}} u.
\end{equation}
\end{lemma}
\begin{proof}
This is a direct application of the Harnack's inequality of \cite[Theorem 8.20]{GT83}.
\end{proof}

\bibliographystyle{alpha}
\bibliography{Obs2D}

\end{document}